\documentstyle{article}

\font\fraktwelve=eufm10 at 12pt
\font\frakten=eufm10
\font\frakseven=eufm7

\def\eul{\fam=9}

\newcommand{\Gp}{{\eul p}}
\newcommand{\Gq}{{\eul q}}

\newcommand{\Gs}{{\eul s}}
\newcommand{\ssigma}{{\!\! \mbox{ }^\sigma \!}}
\newcommand{\ssigmai}{{\!\! \mbox{ }^{\sigma_{i}} \!}}
\newcommand{\ssigmaone}{{\!\! \mbox{ }^{\sigma_{1}}}}
\newcommand{\ssigmar}{{\!\! \mbox{ }^{\sigma_{r}} \!}}
\newcommand{\ssigman}{{\!\! \mbox{ }^{\sigma_{n}}}}

\catcode`\@=11

\def\section{\@startsection {section}{1}{\z@}{-3.5ex plus
  -1ex minus -.2ex}{2.3ex plus .2ex}{\normalsize\bf}}
\def\subsection{\@startsection {subsection}{2}{\z@}{-3.25ex plus
  -1ex minus -.2ex}{1.5ex plus .2ex}{\normalsize\bf}}

\catcode`\@=12

\newtheorem{thm}{Theorem}[section]
\newtheorem{prop}[thm]{Proposition}
\newtheorem{cor}[thm]{Corollary}
\newtheorem{lemma}[thm]{Lemma}

\newcommand{\blacksquare}{{\vrule height2ex width0.75em}}

\newenvironment{proof}{\begin{sc}\noindent Proof: \end{sc}}{
     \hbox to 2em{}\nobreak\hfill$\blacksquare$\par\medskip}

\newcommand{\Gal}{\mathop{\rm Gal}}

\newfont{\cyr}{wncyr10}
\newcommand{\Sh}{\hbox {\cyr Sh}}

\begin{document}

\normalsize
\baselineskip=18pt

\begin{center}
{\bf Computing a Selmer group of a Jacobian using functions on the curve} \\
Mathematische Annalen, 310, 447--471, (1998). \\
This version contains the correction to Proposition 2.4.
\end{center}

\vspace{.4in}

\begin{center}
Edward F.\ Schaefer \\
Santa Clara University
\end{center}

\noindent
Math Subject Classification number (primary): 11G30, (secondary):
 11D25,
11G10, 14G25, 14H25, 14H40, 14H45

\vspace{.4in}

\noindent
{\bf Abstract}\footnote{{\bf Acknowledgments:}
The author is supported by the National Security
Agency grant MDA904-95-H-1051.
The author is
grateful to Gerhard Frey and the Institut f\"{u}r
experimentelle Mathematik for their hospitality during the preparation of
part of this paper and to Joseph Wetherell and the referee
for helpful comments on the manuscript. The author also had useful discussions
with Victor Flynn,
Everett Howe, Matthew Klassen and Michael Stoll and made much use of
the program GP-PARI. The author is grateful to Adam Logan for pointing out the
error in Proposition 2.4 (now corrected).}

In general, algorithms for computing the Selmer group of the Jacobian of a
curve have relied on either homogeneous spaces or functions on the curve.
We present a theoretical analysis of algorithms which use functions on the
curve, and show how to exploit special properties of curves to
generate new Selmer group computation algorithms.  The success of such an
algorithm will be based on two criteria that we discuss.
To illustrate the types of properties which can be exploited, we develop a
$(1-\zeta_{p})$-Selmer group computation algorithm
for the Jacobian of a curve of the form
$y^{p}=f(x)$ where $p$ is a prime not dividing the degree of $f$. We compute
Mordell-Weil ranks of the Jacobians of three curves of this form.
We also compute a
2-Selmer group for the Jacobian of a smooth plane
quartic curve using bitangents of that curve, and use it to compute a
Mordell-Weil rank.

\section{Introduction}
\label{intro}	

Several algorithms have been developed for computing Selmer
groups for the Jacobians of curves. Typically,
one is interested in computing a Selmer group in order to bound a
Mordell-Weil rank or study a part of a Tate-Shafarevich group
(see, for example, \cite{Kr}).
For curves of genera 1 and 2,
algorithms using homogeneous spaces have been developed for
computing Selmer groups (\cite{BSD,GG}).  Already in the genus 2
case, the homogeneous spaces are quite difficult to describe.
For that reason, these tend to be somewhat unwieldy to implement.
Other algorithms use functions on the curve to
compute a Selmer group (\cite{BK,Ca,CF,Fd,FPS,KS,Mc,PS,Sc1,Tp}).
These tend to be far easier.
Their success, however, seems to be based on two assumptions.
These assumptions have been satisfied in the examples presented, so far,
but should not be expected to be satisfied in all
cases.
In this paper,
we attempt to provide a framework for the study and development of
algorithms for computing Selmer groups using functions on the curve.
In particular, we consider the assumptions they are based on.

Let $C$ be a curve defined over the number field $K$ and let $J$
be its Jacobian. We standardly identify $J$ with ${\rm Pic}^{0}
(C(\overline{K}))$ which we will denote ${\rm Pic}^{0}(C)$.
Let $A$ be an abelian variety defined over $K$ and let $\phi :A\rightarrow
J$ be an isogeny defined over $K$. Let $A[\phi]$ denote the kernel of
$\phi$.
For most practical purposes (such as descent),
it is really only useful to work with an
isogeny whose kernel has a prime-power exponent. So
we assume that $A[\phi]$ has
exponent $q=p^{l}$ for some prime number $p$.

In Section~\ref{alg} we provide a framework for developing an
algorithm for computing the $\phi$-Selmer group for $A$ over $K$.
For the sake of clarity, we first describe a straightforward
way of creating such an algorithm.
The assumptions that such an algorithm is based on
are discussed in Section~\ref{ass}.
All but one algorithm in the literature, that the author is aware of,
fits into the framework described. In Section~\ref{ext},
we discuss how this framework can be extended in special cases
to encompass this and other algorithms.

The strength of this approach is that it allows us to develop an
algorithm tailored to the data at hand. We give several examples.
In Section~\ref{ytop}
we describe an algorithm
for computing a Selmer group
for curves of the form $y^{p}=f(x)$ where $p$ is
a prime not dividing the degree of
$f$. We do three examples. Let $\zeta_{p}$ denote a primitive
$p$th root of unity.
In the first, we find the
Mordell-Weil ranks of the Jacobian of $y^{3}=(x^{2}+1)(x^{2}-4x+1)$
over ${\bf Q}(\zeta_{3})$ and ${\bf Q}$. In the second we describe
all solutions of $y^{2}=x^{5}+1$ in fields of degree 2 or less over
${\bf Q}$. In the third we describe all solutions of
$y^{3}=x(x-1)(x-2)(x-3)$ in fields of degree 3 or less over
${\bf Q}$. In the latter two examples, the Mordell-Weil ranks of
the Jacobians are 0 over ${\bf Q}(\zeta_{5})$ and ${\bf Q}(\zeta_{3})$
respectively.
In Section~\ref{bitan} we
compute the 2-Selmer group and Mordell-Weil rank, over ${\bf Q}$,
of the Jacobian of a smooth plane quartic curve, using
bitangents of the curve.

In Section~\ref{lit} we give a review of the literature in which
algorithms for curves of genus greater than 1 are discussed.

\section{The algorithm}
\label{alg}

Let us define the Selmer group.
Let $J$, $K$, $A$, $\phi$, $q$ and $p$ be as in section~\ref{intro}.
Let $S$ be a finite set of primes of $K$ that includes primes over $p$,
primes dividing the conductor of $A$, and if $p=2$, includes real primes
also. For any $\Gal (\overline{K}/
K)$-module $M$ let $M(K)$ denote the $\Gal (\overline{K}/K)$-invariants of $M$
and $H^{1}(K,M)$ denote $H^{1}(\Gal (\overline{K}/K),M)$.
Let $H^{1}(K,A[\phi];S)$ denote the subgroup of $H^{1}(K,A[\phi])$
of cocycle classes that are unramified outside $S$.
Let $\delta$ be the
map from $J(K)$ to $H^{1}(K,A[\phi])$
arising from the long exact sequence of Galois cohomology attached to the
short exact sequence
\[
0\rightarrow A[\phi] \rightarrow A \stackrel{\phi}{\rightarrow} J \rightarrow
0.\]
The kernel of $\delta$ is $\phi A(K)$.
Similarly, for any prime $\Gs\in S$ we have a coboundary map
$\delta_{\Gs}$
{}from $J(K_{\Gs})$ to $H^{1}(K_{\Gs},A[\phi])$ with kernel $\phi A(K_{\Gs})$.
Let $\alpha_{\Gs}$ be
the restriction map from $H^{1}(K,A[\phi])$ to $H^{1}(K_{\Gs},A[\phi])$.
The following is a commutative diagram.
\[
\begin{array}{ccc}
J(K)/\phi A(K) & \stackrel{\delta}{\hookrightarrow} & H^{1}(K,A[\phi];S) \\
\downarrow & & \downarrow\prod\alpha_{\Gs} \\
\prod\limits_{\Gs\in S}
J(K_{\Gs})/\phi A(K_{\Gs}) & \stackrel{\prod\delta_{\Gs}}{\hookrightarrow}
 & \prod\limits_{\Gs\in S}H^{1}(K_{\Gs},A[\phi])\end{array}\]
Define the Selmer group, $S^{\phi}(K,A)$,
to be the intersection of the groups $\alpha_{\Gs}^{-1}(\delta_{\Gs}
(J(K_{\Gs})/\phi A(K_{\Gs})))$ for
all $\Gs\in S$. This is equivalent to the usual definition (see
\cite[p.\ 92]{Mi3}).

In Section~\ref{alg1}, we
find a finitely generated $K$-algebra $L$ and a
map $F$, derived from functions on $C$, so that $F$ maps
$J(K)/\phi A(K)$ to $L^{\ast}/L^{\ast q}$.
In Section~\ref{eom}
we find a map $\iota$ from
$H^{1}(K,A[\phi])$ to $L^{\ast}/L^{\ast q}$ so that
$F=\iota\circ \delta$. The map $\iota$ will be induced from a Weil pairing
and a Kummer map.
Let $L_{\Gs}=L\otimes_{K}K_{\Gs}$.
We will similarly be able to define maps
$F_{\Gs}$ and  $\iota_{\Gs}$ so that $F_{\Gs}=\iota_{\Gs}\circ \delta_{\Gs}$.
Once these maps are defined,
the following will be a commutative diagram
where the $\beta_{\Gs}$'s
are natural maps.
\[
\begin{array}{ccccc}
J(K)/\phi A(K) & \stackrel{\delta}{\hookrightarrow} & H^{1}(K,A[\phi];S)
 & \stackrel{\iota}{\hookrightarrow}  & L^{\ast}/L^{\ast q} \\
\downarrow & & \downarrow\prod\alpha_{\Gs} & & \downarrow\prod\beta_{\Gs} \\
\prod\limits_{\Gs\in S}
J(K_{\Gs})/\phi A(K_{\Gs}) & \stackrel{\prod\delta_{\Gs}}{\hookrightarrow}
 & \prod\limits_{\Gs\in S}H^{1}(K_{\Gs},A[\phi]) &
\stackrel{\prod\iota_{\Gs}}{\hookrightarrow}  &
\prod\limits_{\Gs\in S} L_{\Gs}^{\ast}/L_{\Gs}^{\ast q}
\end{array}\]
We finish that section by making an assumption causing
$\iota$ and $\iota_{\Gs}$ and
hence $F$ and $F_{\Gs}$ to be injective.
In Section~\ref{alg2},
we show how to use
the maps $F$ and $F_{\Gs}$ to compute the Selmer group.

In order for the maps $F$ and $F_{\Gs}$ to be
derived from
functions on $C$, we need to make the following assumption.
We will denote ${\rm Div}^{0}(C(\overline{K}))$ by ${\rm Div}^{0}(C)$.

\vspace{1mm}
\noindent
{\bf Assumption I}: For ${\cal K}= K$ or $K_{\Gs}$, with $\Gs\in S$,
every element of $J({\cal K})/\phi A({\cal K})$ is represented by a
divisor class containing an element of ${\rm Div}^{0}(C)({\cal K})$,
the divisors of $C$ of degree 0 defined over ${\cal K}$.

\subsection{The choice of $F$ and $L$}
\label{alg1}

Since we will be dealing with a Weil pairing, we need to consider the
dual isogeny to $\phi$.
Let $\hat{\phi} : \hat{J} \rightarrow \hat{A}$ be the dual isogeny
to $\phi$ and $\hat{J}[\hat{\phi}]$ be the kernel of $\hat{\phi}$.
Let $\lambda : J\rightarrow \hat{J}$
be the canonical principal polarization of $J$ with respect to $C$.
Since $C$ is defined over $K$, the principal polarization $\lambda$ is also
(see \cite[p.\ 186]{Mi1}).
Let $\Psi=\lambda^{-1}(\hat{J}[\hat{\phi}])$;
we know $\Psi$ is contained in $J[q]$.

\vspace{3mm}
\noindent
{\bf Step 1.} Determine the subgroup of $J[q]$ that is $\Psi$.

If $\phi$ is the multiplication by $q$ map, then $\Psi=J[q]$. For
non-trivial examples, see Proposition~\ref{pp} and Section~\ref{lit}.

\vspace{3mm}
\noindent
{\bf Step 2.} Choose some suitable $\Gal (\overline{K}/
K)$-invariant set of divisors in ${\rm Div}^{0}(C)$ whose
classes span $\Psi$.

We denote
the linear equivalence class of the degree-0 divisor $D$
in ${\rm Pic}^{0}(C)$ by $[D]$.
We choose $\{ D_1, \ldots ,D_n \} $ to be a $\Gal
(\overline{K}/K)$-invariant set of degree-0 divisors of $C$ for
which the divisor classes $\{ [D_{i}] \}$ span $\Psi$.
As we will see, the choice of spanning set determines the map
$F$ and the $K$-algebra $L$.
Typically one wants a minimal, Galois-invariant spanning set.
We also want to pick a spanning set so that the second assumption holds,
if possible. This is discussed in Section~\ref{ass}.

\vspace{3mm}
\noindent
{\bf Step 3.} Determine the map $F$ and the finitely generated $K$-algebra
$L$ based on the divisors chosen in Step 2.

First we define the finitely generated $K$-algebra $L$.
Let \[
L'=\prod_{i=1}^{n}\overline{K_{i}}\]
with $K_{i}=K$. Let us define an action of $\Gal (\overline{K}/K)$ on
$L'$.
If $\sigma\in \Gal (\overline{K}/K)$, then
let $\overline{\sigma}\in S_{n}$ be defined such that if $\ssigma D_{i}
= D_{j}$, then $\overline{\sigma}i=j$. Let \[
\ssigma (a_{1},\ldots ,a_{n}) = (\ssigma a_{\overline{\sigma}^{-1}1},
\ldots ,\ssigma a_{\overline{\sigma}^{-1}n})\] for $a_{i}\in
\overline{K_{i}}$. Define $L$ to be the $\Gal (\overline{K}/K)$-invariants
in $L'$.

Let us find a more practical description of $L$. Let $\Lambda$
be a subset of $\{ 1, 2, \ldots ,n \} $ such that the set $\{ D_{j} \}_{j\in
\Lambda}$ contains one representative of each $\Gal (\overline{K}/K)$-orbit
of $\{ D_{i} \} $. Let $L_{j}=K(D_{j})$ be the minimal field of definition of
$D_{j}$. Then we can find an isomorphism
\[
L\cong \prod_{j\in\Lambda}L_{j}.\] Let us describe that isomorphism.
For simplicity, assume that $\Gal (\overline{K}/K)$ acts transitively
on the $\{ D_{i} \} $ and let $\Lambda = 1$. We have $L_{1}=K(D_{1})$.
Let $\{ \sigma_{i} \} $ be elements of $\Gal (\overline{K}/K)$ such that
\[
\Gal (\overline{K}/K) = \coprod_{i} \sigma_{i} \Gal (\overline{K}/L_{1})\]
and
$\ssigmai D_{1}=D_{i}$.
We have\[
L_{1}\cong L\;\; {\rm by}\;\; l\in L_{1}\mapsto (\ssigmaone l,\ldots ,
\ssigman l)\in \prod_{i=1}^{n}\overline{K_{i}}.\]
If there are several orbits, then we can extend this isomorphism by
concatenation.

Let us define the map $F$.
Let $qD_{i}=(f_{i})$ where $f_{i}$ is defined over $K(D_{i})$.
Such $f_{i}$'s exist over $K(D_{i})$ by Hilbert's Theorem 90.
Let ${\rm Supp}(D_{i})$ be the support of the divisor
$D_{i}$.

\vspace{1mm}
\noindent
{\bf Definition:} The avoidance set is the set of points
$\cup{\rm Supp}(D_{i})$ in
$C(\overline{K})$.

\noindent
Define $F=(f_{1},\ldots
,f_{n})$ from the complement of the avoidance set to $L'$.
By abuse of notation we use $P_{L}$ to denote the $n$-tuples of divisors
and divisor classes $(D_{1},\ldots ,D_{n})$ and $([D_{1}],\ldots ,[D_{n}])$.
Let $\hat{P}_{L}$ denote the $n$-tuple of elements
$(\lambda [D_{1}],\ldots ,\lambda [D_{n}])$
of $\hat{J}[\hat{\phi}]$.

\subsection{Equivalence of maps}
\label{eom}

In this subsection we
show that $F$ induces a well-defined homomorphism from $J(K)/\phi
A(K)$ to $L^{\ast}/L^{\ast q}$ and that $F$ is related to cohomological
maps used to define a Selmer group. We will return to the algorithm in
subsection~\ref{alg2}.

\vspace{1mm}
\noindent
{\bf Definition:} A good divisor is a divisor of $C$ of degree 0,
defined over $K$ (or $K_{\Gs}$), whose support does not intersect
the avoidance set.

{}From Assumption I, every element of $J(K)/\phi A(K)$ is
represented by a divisor class containing
a divisor of degree 0, defined over $K$.
{}From \cite[Lemma 3, p.\ 166]{La},
every divisor class that contains a divisor defined over $K$, contains a
divisor defined over $K$, whose
support does not intersect any given finite set, in particular, the
avoidance set. So every element of $J(K)/\phi A(K)$ is represented
by a good divisor.

Let $g$ be a $K$-defined function from $C$ to $\overline{K}$.
Let $R=\sum n_{i}R_{i}$ be a divisor of $C$ of degree 0, defined over
$K$, whose support does not intersect the support of $(g)$.
We define \[
g(R)=\prod (g(R_{i}))^{n_{i}}\in K^{\ast}.\]
By abuse of
notation we define the map $F$ from good divisors to $L^{\ast}$ in an
analogous way.

The map $F$ on good divisors, composed with the isomorphism of
$L$ with $\prod_{j\in\Lambda}L_{j}$, is the map
$\prod_{j\in\Lambda}f_{j}$. In examples, we will often denote
$\prod_{j\in\Lambda}L_{j}$ by $L$ and this composition by $F$, since
they are more practical.

\begin{lemma}
\label{well-def}
The map $F$ induces a homomorphism from the subgroup of
$J(K)/q J(K)$ represented by divisor classes containing
good divisors to $L^{\ast}/L^{\ast
q}$.
\end{lemma}

\begin{proof}
The good divisors form a subgroup of ${\rm Div}^{0}(C)(K)$.
The map $F$ is a homomorphism from good divisors to $L^{\ast}$.
Let $D$ and $D'$ be linearly equivalent good divisors.
We would like to show that $F(D-D')$ is in $L^{\ast q}$.
Let $h$ be a $K$-defined function with $(h)=D-D'$. From
Weil reciprocity, we have the following equalities of $n$-tuples
\[
F(D-D')=F((h))=h((F))=h(qP_{L})=(h(P_{L}))^{q}\in L^{\ast}.\]
Since $P_{L}$ is fixed by $\Gal (\overline{K}/K)$ we know
$h(P_{L})$ is in $L^{\ast}$. So $F(D-D')$ is in $L^{\ast q}$.
\end{proof}

At this point let us
relate the map $F$ to the maps derived from cohomology
which are traditionally used to compute a Selmer group.
First let
us recall the definition of the Weil pairing. Let $\tau$ be an isogeny
of abelian varieties from $B$ to $V$ and $\hat{\tau}$ be the dual isogeny from
$\hat{V}$ to $\hat{B}$. Let $P\in B[\tau]$ and $Q\in \hat{V}[\hat{\tau}]$ and
$D$ be a divisor on $\hat{B}$ representing $P$. There is a function
$g$ on $\hat{V}$ with divisor $\hat{\tau}^{-1}D$. Then $e_{\tau}(P,Q) =
g(X+Q)/g(X)$ for any $X\in \hat{V}$ for which the right hand side of
the equation is defined.

Let $\mu_{q}(L')$ be the $q$th
roots of unity in $L'$.
We have \[
\mu_{q}(L')\cong \mu_{q}
(\overline{K_{1}})\times \ldots \times
\mu_{q}(\overline{K_{n}}).\]
Let $e_{\phi}(P,Q)$ denote the $\phi$-Weil pairing
of $P\in A[\phi]$ and $Q\in \hat{J}[\hat{\phi}]$.
Define the map $w$ from $A[\phi]$ to $\mu_{q}(L')$ by
\[
w(P)=e_{\phi}(P,\hat{P}_{L})=(e_{\phi}(P,\lambda [D_{1}]),\ldots
,e_{\phi}(P,\lambda [D_{n}])).\]

\begin{prop}
\label{weil}
The map $w$ from $A[\phi]$ to $\mu_{q}(L')$ is injective and defined over $K$.
\end{prop}

\begin{proof}
Since the elements of the $n$-tuple $\hat{P}_{L}$ span $\hat{J}[\hat{\phi}]$,
the map $w$ is injective from the non-degeneracy of the Weil pairing.
Let $\sigma\in \Gal (\overline{K}/K)$. We would like to show that
$\ssigma w(P) = w(\ssigma P)$.
We have \[
w(\ssigma P) = (e_{\phi}(\ssigma P,\lambda [D_{1}]),\ldots  ,
e_{\phi}(\ssigma P,\lambda [D_{n}]))\]
and \[
\ssigma w(P) = \ssigma (e_{\phi}(P,\lambda [D_{1}]),\ldots ,e_{\phi}(P,\lambda
[D_{n}])) = (\ssigma e_{\phi}(P,\lambda
[D_{\overline{\sigma}^{-1}1}]),
\ldots ,
\ssigma e_{\phi}(P,\lambda [D_{\overline{\sigma}^{-1}n}]))\]\[
= (e_{\phi}(\ssigma P,\lambda [\ssigma D_{\overline{\sigma}^{-1}1}]),
\ldots ,
e_{\phi}(\ssigma P,\lambda [\ssigma D_{\overline{\sigma}^{-1}n}]))
= (e_{\phi}(\ssigma P,\lambda [D_{1}]),\ldots ,
e_{\phi}(\ssigma P,\lambda [D_{n}])).\]
\end{proof}
The map $w$ induces a map from $H^{1}(K,A[\phi])$ to
$H^{1}(K,\mu_{q}(L'))$ which we also call $w$.
{}From \cite[p.\ 152]{Se}, $H^{1}(K,\mu_{q}
(L'))\cong
L^{\ast}/L^{\ast q}$
by a map we call $k$.
The map $k$ sends the cocycle class containing $(\sigma\mapsto
\ssigma (\sqrt[q]{l})/\sqrt[q]{l})$ to $l\in L^{\ast}$.

\begin{thm}
\label{bigthm}
The maps
$F$ and $k\circ w\circ\delta$ are the same as maps
{}from $J(K)/\phi A(K)$ to $L^{\ast}/L^{\ast q}$.
\end{thm}

\begin{proof}
Let $\delta_{q}$ be the coboundary map from $J(K)/qJ(K)$ to $H^{1}(K,
J[q])$.
Let $w_{q}$ be the map from $J[q]$
to $\mu_{q}(L')$ that sends $R\in J[q]$ to
$e_{q}(R,\hat{P}_{L})$. For clarity we redenote $\delta$ and
$w$ by
$\delta_{\phi}$ and $w_{\phi}$.
We first show that the map $F$ is the same as
the composition $k\circ w_{q}\circ \delta_{q}$
on the subgroup of $J(K)/qJ(K)$ of elements represented by
divisor classes containing good divisors.
Let $P$ be a good divisor representing such an element of $J(K)/qJ(K)$.
{}From
Lemma~\ref{well-def}, the choice of such a $P$ is unimportant.
{}From \cite[Lemma 3, p.\ 166]{La}, we can pick a degree 0 divisor $Q$,
whose support does not intersect the
avoidance set, and for which $qQ$ is linearly equivalent to $P$. The
class of cocycles $\delta_{q} ([P])$ includes the cocycle
$(\sigma\mapsto [\ssigma Q-Q])$ with $\sigma\in\Gal (\overline{K}/K)$.
So $w_{q}\circ
\delta_{q} ([P])$ is
the class of cocycles that includes $(\sigma\mapsto
e_{q}([\ssigma Q-Q],\hat{P}_{L}))$.
Let
$e_{q}^{\lambda}(S,T)=e_{q}(S,\lambda T)$ for $S,T\in J[q]$.
The $e_{q}^{\lambda}$-Weil pairing can be defined as follows. If
$h_{1}$ and $h_{2}$ are functions on $C$ with divisors $qE_{1}$ and
$qE_{2}$ respectively, with disjoint supports,
then $e_{q}^{\lambda}([E_{1}],[E_{2}])=
h_{2}(E_{1})/h_{1}(E_{2})$.
Let $(g)=qQ-P$ with $g$ defined over $K(Q)$.
We have $(\ssigma g)=q\ssigma Q-P$ and
so $(\ssigma g/g)=q\ssigma Q-qQ$.
Recall $(f_{i})=qD_{i}$.
We have
\[
e^{\lambda}_{q}([\ssigma Q-Q],[D_{i}])
=\frac{f_{i}(\ssigma Q-Q)}
{\ssigma g/g(D_{i})}.\]
Thus we have the following equalities of $n$-tuples \[
e_{q}([\ssigma Q-Q],\hat{P}_{L})=
\frac{F(\ssigma Q-Q)}{\ssigma g/g(P_{L})}
=\frac{\ssigma \beta}{\beta}\]
where $\beta=F(Q)/g(P_{L})$. So we have\[
k\circ w_{q}\circ
\delta_{q} ([P])\equiv \beta^{q} \equiv \frac{F(qQ)}
{g(qP_{L})}\equiv \frac{F(qQ)}{F(qQ-P)}\equiv F(P)\; {\rm mod}\;
L^{\ast q}.\]

Let us show that $F$ and $k\circ w_{\phi}\circ\delta_{\phi}$ are
the same as maps from $J(K)/\phi A(K)$ to $L^{\ast}/L^{\ast q}$.
It follows from Assumption I that every element of $J(K)/\phi A(K)$
is represented by a good divisor.
There is an isogeny $\tau : J\rightarrow A$ with $\phi\circ\tau =q$.
{}From the
commutative diagram\[
\begin{array}{ccccccccc}
0 & \rightarrow & J[q] & \rightarrow & J(\overline{K}) & \stackrel{q}
{\rightarrow} & J(\overline{K}) & \rightarrow & 0 \\
 & & \downarrow \tau & & \downarrow \tau & & \downarrow 1 & & \\
0 & \rightarrow & A[\phi] & \rightarrow & A(\overline{K}) & \stackrel
{\phi}{\rightarrow} & J(\overline{K}) & \rightarrow & 0 \end{array}\]
we get the following commutative diagram by taking $\Gal (\overline{K}/
K)$-invariants.
\[
\begin{array}{ccc}
J(K)/qJ(K) & \stackrel{\delta_{q}}{\rightarrow} & H^{1}(K,J[q]) \\
\downarrow & & \downarrow \tau \\
J(K)/\phi A(K) & \stackrel{\delta_{\phi}}{\rightarrow} & H^{1}(K,A[\phi])
\end{array}\]

{}From the compatibility of Weil pairings
we have $e_{q}(R,\hat{P}_{L})=
e_{\phi}(\tau(R),\hat{P}_{L})$. Thus the
triangle of the following diagram commutes and so the whole diagram
commutes.
\[
\begin{array}{ccccccc}
J(K)/qJ(K) & \stackrel{\delta_{q}}{\rightarrow} &
H^{1}(K,J[q]) &
 \searrow w_{q} & & & \\
\downarrow
 & & \downarrow \tau & &
H^{1}(K,\mu_{q}(L'))
& \stackrel{k}{\rightarrow} & L^{\ast}/
L^{\ast q} \\
J(K)/\phi A(K) & \stackrel{\delta_{\phi}}{\rightarrow} &
H^{1}(K,A[\phi]) & \nearrow w_{\phi} & & &
\end{array}\]
{}From commutivity, $F$ must factor through $\phi A(K)$ and
$F$ and $k\circ w_{\phi}\circ\delta_{\phi}$ are the same as maps
{}from $J(K)/\phi A(K)$ to $L^{\ast}/L^{\ast q}$.
\end{proof}

\noindent
Inspiration for this proof can be found in \cite{Li} and
\cite[Lemma 2.2]{Mc}.

Note that Lemma~\ref{well-def} and Theorem~\ref{bigthm} hold if we
replace $K$ and $L'$ by any field ${\cal K}$ (containing $K$)
and $L'\otimes_{K}{\cal K}$.
Let $\Gs\in S$ and $L_{\Gs}=L\otimes_{K}K_{\Gs}$. The map $F$ induces
a map $F_{\Gs}$ from $J(K_{\Gs})/\phi A(K_{\Gs})$ to $L_{\Gs}^{\ast}/
L_{\Gs}^{\ast q}$.
In order to compute the Selmer group, we want the maps $F$ and
$F_{\Gs}$ to be injective. The maps $\delta$,
$\delta_{\Gs}$, $k$ and $k_{\Gs}$ are injective automatically.
Let us make an assumption that makes $w$ and $w_{\Gs}$
injective also.

Let ${\cal K}$ be any field containing $K$ and $\overline{\cal K}$ be an
algebraic closure.
Let coker be defined to make the following an
exact sequence of $\Gal (\overline{\cal K}/{\cal K})$-modules. \[
0\rightarrow A[\phi]\stackrel{w}{\rightarrow} \mu_{q}(L')
\rightarrow {\rm coker} \rightarrow 0\]
In addition, let ${\cal K}'$ denote the minimal field of
definition, over ${\cal K}$, of the $D_{i}$'s.

We have
\[
H^{1}({\cal K},A[\phi])\stackrel{w}{\rightarrow}
H^{1}({\cal K},\mu_{q}(L'))\; {\rm is}\;
{\rm injective}\]
\[
\Leftrightarrow \mu_{q}(L')({\cal K})\rightarrow {\rm coker}({\cal K})\;
{\rm is} \; {\rm surjective} \]
\[
\Leftrightarrow H^{1}(\Gal ({\cal K}'/{\cal K}),
A[\phi])\stackrel{w}{\rightarrow}
H^{1}(\Gal ({\cal K}'/{\cal K}),\mu_{q}(L'))\; {\rm is} \; {\rm injective.}\]

\vspace{3mm}
\noindent
{\bf Assumption II}:
The maps $H^{1}(G,A[\phi])\stackrel{w}{\rightarrow}
H^{1}(G,\mu_{p}(L'))$ are injective for $G=\Gal (K'/
K)$ and $G=\Gal (K_{\Gs}'/
K_{\Gs})$ with $\Gs\in S$.

\noindent
This assumption guarantees that
the maps $F$ and $F_{\Gs}$ are injections.

\subsection{Computing the Selmer group}
\label{alg2}

\vspace{3mm}
\noindent
{\bf Step 4.} Find a set $S$

The set of primes of $K$ denoted $S$ must include
the primes dividing the conductor of $A$,
the primes over $p$,
and if $p=2$, the real primes.
The primes dividing the conductor of $A$ are the same as those
dividing the conductor of $J$. These are a subset of the primes at which
the reduction of $C$ is singular.
It is easier, in general, to determine the primes at which the reduction
of $C$ is singular
(see \cite[chap.\ 1, \S 5]{Ha}), than the
primes dividing the conductor of
$J$. So, for simplicity, we can
include in $S$ all of the primes at which the reduction of $C$ is singular.

\vspace{3mm}
\noindent
{\bf Step 5.} Determine the image of $H^{1}(K,A[\phi];S)$ in
$L^{\ast}/L^{\ast q}$ and find generators of that image.

We have $L\cong \prod_{j\in\Lambda} L_{j}$ where the $L_{j}$ are fields.
Thus $L^{\ast}/L^{\ast q}$
is isomorphic to $\prod_{j\in\Lambda} L_{j}^{\ast}/L_{j}^{\ast q}$.

\vspace{1mm}
\noindent
{\bf Definition:}
Let $L_{j}(S,q)$ be
the subgroup of $L_{j}^{\ast}/L_{j}^{\ast q}$ of elements with
the property that if we adjoin the $q$th root of a representative
to $L_{j}$, that we get an extension unramified outside of primes
over primes of $S$.
Let $L(S,q)=\prod_{j\in\Lambda} L_{j}(S,q)$.

\noindent
Since we are making Assumption II, we have \[
H^{1}(K,A[\phi])\cong {\rm ker}:H^{1}(K,\mu_{q}(L'))\rightarrow
H^{1}(K,{\rm coker})\]
and
\[
H^{1}(K,A[\phi];S)\cong {\rm ker}:H^{1}(K,\mu_{q}(L');S)\rightarrow
H^{1}(K,{\rm coker})\cong {\rm ker}:L(S,q)\rightarrow H^{1}(K,{\rm coker}).\]

By abuse of notation, we refer to the subgroup of $L(S,q)$ above
as $H^{1}(K,A[\phi];S)$.
Let $\beta_{\Gs}$ be the natural map from $L^{\ast}/L^{\ast q}$ to
$L_{\Gs}^{\ast}/L_{\Gs}^{\ast q}$.
The image of $J(K)/\phi A(K)$ in $H^{1}(K,A[\phi])$ actually
lies in $H^{1}(K,A[\phi];S)$ and the following diagram commutes
(see \cite[p.\ 92]{Mi3}).
\[
\begin{array}{ccc}
J(K)/\phi A(K) & \stackrel{F}{\hookrightarrow} & H^{1}(K,A[\phi];S) \\
\downarrow & & \downarrow\prod\beta_{\Gs} \\
\prod\limits_{\Gs\in S}
J(K_{\Gs})/\phi A(K_{\Gs}) & \stackrel{\prod F_{\Gs}}{\hookrightarrow}
 & \prod\limits_{\Gs\in S}L_{\Gs}^{\ast}/L_{\Gs}^{\ast q}\end{array}\]
{}From Theorem~\ref{bigthm} and the injectivity of $k\circ w$
the Selmer group, $S^{\phi}(K,A)$,
is isomorphic to
the intersection of the groups
$\beta_{\Gs}^{-1}(F_{\Gs}
(J(K_{\Gs})/\phi A(K_{\Gs})))$ for
all $\Gs\in S$.

\vspace{3mm}
\noindent
{\bf Step 6.}
Find generators for $J(K_{\Gs})/\phi A(K_{\Gs})$ and their images
under $F_{\Gs}$, in $L_{\Gs}^{\ast}/L_{\Gs}^{\ast q}$, for all $\Gs\in S$.

For representatives, we find $K_{\Gs}$-rational, degree 0 divisors. It may
be necessary to shift their supports so that we have good divisors.
To check
if the classes of good divisors are independent in $J(K_{\Gs})/\phi
A(K_{\Gs})$, it is easiest to check if their images under the injective
map $F_{\Gs}$ are independent in $L_{\Gs}^{\ast}/L_{\Gs}^{\ast q}$.
A deterministic algorithm for finding such generators for the Jacobians
of curves of genus 2 and the multiplication by 2 map is given in \cite{St1}.

We need to know how many generators are needed.
Let $C$ have genus $g$ and $\Gs$ be a finite prime of $K$.
Recall $q=p^{l}$. 
For some sufficiently large $m$, the neighborhoods $A_m (K_\Gs)$ and $J_m (K_\Gs)$
of the 0-points are isomorphic and the isogeny $\phi$ can be written as a $g$-tuple of power
series in $g$ variables. 
For elements $k\in K_\Gs$, we define the normalized absolute value $|k|$ such that if $\pi$ is a prime of $K_\Gs$ and ${\bf F}_\Gs$ is the residue class field of $K_\Gs$, then
$|\pi|=(\# {\bf F}_{\Gs})^{-1}$.  Let $|\phi'(0)|$ be the normalized absolute value
of the determinant of the Jacobian matrix associated to the above power series
for $\phi$, evaluated at the 0-point.
Let $c_A$ and $c_J$ be the Tamagawa numbers of $A$ and $J$.

\begin{prop}
\label{size}
Assume $\ell$ is a prime number and $\Gs | \ell$. Let $r={\rm ord}_\ell (q)$.
Then
$\# J(K_{\Gs})/q J(K_{\Gs}) = p^{gr[K_{\Gs}:{\bf Q}_{p}]}\cdot \#
J(K_{\Gs})[q]$.
More generally, $\# J(K_{\Gs})/\phi A(K_{\Gs})$
$=|\phi'(0)|^{-1}\cdot \# A(K_\Gs)[\phi]\cdot c_J/c_A$.
\end{prop}

Both statements can be shown using the snake lemma and the fact that
$J(K_{\Gs})$ contains a subgroup of finite index isomorphic to $g$
copies of the ring of integers in $K_{\Gs}$ (see \cite{Ma} and
\cite[Lemma 3.8, Prop. 3.9]{Sc2}).
Of course, $\# J(\overline{K_{\Gs}})[q]=q^{2g}$.
If $\phi$ is not a multiplication
by $q$ map, then the computation of
$\# J(K_{\Gs})/\phi A(K_{\Gs})$ is not always trivial. This is
discussed in \cite[\S 3]{Sc2}, where an algorithm is given, in the
case that $J$ is an elliptic curve.  In certain other cases it can be
accomplished, as in Corollary~\ref{localsize}.

If $p=2$, then $S$ includes real primes.

\begin{prop}
\label{sizeR}
If $q=2^{l}$ and
$J$ is defined over ${\bf R}$, then
$\# J({\bf R})/qJ({\bf R})= q^{-g}\cdot \# J({\bf R})[q]$.
\end{prop}
For the proof, simply replace 2 by $q=2^{l}$ in
\cite[Prop.\ 5.4]{Sc2}
where one can find
discussions of more general isogenies of even degree at real
primes.

\vspace{2mm}

\noindent
{\bf Step 7.} Find the intersection in $H^{1}(K,A[\phi];S)$
of $\beta_{\Gs}^{-1}(F_{\Gs}(J(K_{\Gs})/
\phi A(K_{\Gs})))$ for all $\Gs\in S$.

At this point we have accomplished our goal of computing the Selmer group.
One reason that Selmer groups are computed is for the purpose of
bounding the Mordell-Weil rank.
The group $J(K)$ is called the Mordell-Weil group of $J$ over $K$.
It is a finitely generated abelian group and its free {\bf Z}-rank
is called the Mordell-Weil rank of $J$ over $K$.
In order to find the Mordell-Weil rank, we need to find
elements of $J(K)/\phi A(K)$ and map them
to $L^{\ast}/L^{\ast q}$.
One can save time by doing this before Step 6
since elements of $J(K)/\phi A(K)$ map to elements in each $J(K_{\Gs})/
\phi A(K_{\Gs})$.
Let $\Sh (K,A)[\phi]$ denote the
$\phi$-torsion of the Tate-Shafarevich group for $A$ over $K$.
We hope to generate all of the kernel from
$S^{\phi}(K,A)$ to
$\Sh (K,A)[\phi]$ (assuming you know what $\Sh (K,A)[\phi]$ is!) This
kernel is isomorphic to $J(K)/\phi A(K)$. If we have success, then
we can attempt to compute the
Mordell-Weil rank of $J(K)$.

Let
$\phi'$ be an isogeny from $J$ to $A$ for which
$\phi\circ\phi'=\tau$ and $\tau^{t}=mu$ for some unit $u$ in
${\rm End}(J)$,
and integers $t$ and $m=q^{j}$. In addition, assume that
we are able to compute $A(K)/\phi' J(K)$ (at this point it would
be helpful if $A$ were a Jacobian).
The following proposition contains an exact sequence which helps
combine the sizes of these groups to find the size of $J(K)/mJ(K)$.

\begin{prop}
\label{snake}
Let $B$ and $D$ be abelian groups and let $f:B\rightarrow
D$ and $g:D\rightarrow B$ be homomorphisms. The following is
an exact sequence
\[
0 \rightarrow B[f]/g(D[fg]) \rightarrow B/gD \stackrel{f}{\rightarrow}
D/fgD \rightarrow D/fB \rightarrow 0.\]
\end{prop}
\begin{proof}
The proposition follows from the diagram below, which commutes from
the snake lemma applied to the middle two exact sequences.
\[
\begin{array}{ccccccccc}
0 & \rightarrow & g(D[fg]) & \rightarrow & B[f] & \rightarrow &
B[f]/g(D[fg]) & \rightarrow & \\
 & & \downarrow & & \downarrow & & \downarrow & & \\
0 & \rightarrow & gD & \rightarrow & B & \rightarrow & B/gD & \rightarrow
& 0 \\
 & & \;\;\;\downarrow f & & \;\;\;\downarrow f & & \;\;\;\downarrow f & & \\
0 & \rightarrow & fgD & \rightarrow & D & \rightarrow & D/fgD & \rightarrow
& 0 \\
 & & \downarrow & & \downarrow & & \downarrow & & \\
 & \rightarrow & 0 &
 \rightarrow & D/fB & \rightarrow & D/fB & \rightarrow & 0 \end{array}
\]
\end{proof}

If $f$ and $g$ are isogenies of abelian varieties, then the group
$B[f]/g(D[fg])$ will be a quotient of torsion groups, hence computeable.
Replacing $f$ and $g$ by $\phi$ and $\phi'$,
{}from $J(K)/\phi A(K)$ and $A(K)/\phi' J(K)$ we can compute
$J(K)/\tau J(K)$ using the exact sequence.
By replacing $f$ and $g$ by $\tau^{i}$, $1\leq i \leq t-1$, and $\tau$,
we can compute the size of
$J(K)/mJ(K)$. If $r$ is the Mordell-Weil rank of $J$ over $K$, then
$J(K)/mJ(K)\cong ({\bf Z}/m{\bf Z})^{r}\oplus J(K)[m]$.

On the other hand, $\phi$-Selmer groups have many interesting uses
beyond computing $J(K)/mJ(K)$. For example, in
\cite{We}, Wetherell provides a method of bounding the number of
rational points on a curve $C$ when the Mordell-Weil rank is at least
as large as the genus. This is the case that effective Chabauty (see
\cite{Co})
does not help with. In order to do this, he considers a set of covers
of the curve parametrized by elements of a $\phi$-Selmer group,
where $\phi$ is any isogeny from an abelian variety to the Jacobian of $C$.

\subsection{Assumptions}
\label{ass}

Let us consider the two assumptions that we made.

Let $d={\rm gcd}\, \{ \, [K^{\dagger}:K]\; |\; K\subseteq K^{\dagger}
\subset \overline{\bf Q}, \;
C(K^{\dagger})\neq \emptyset\} $. Recall that the exponent of
$A[\phi]$ is $q=p^{l}$.

\begin{prop}
\label{assi}
If $p\! \not | d$, then Assumption I is satisfied.
\end{prop}

\begin{proof}
Let $K^{\dagger}$
be a field of degree $d^{\dagger}$
over $K$ with $C(K^{\dagger})\neq \emptyset$.
{}From \cite[p.\ 168]{Mi1},
every element of $J(K^{\dagger})$ is represented by
a degree 0 divisor of $C$ defined over $K^{\dagger}$. Let $N$ denote any map
induced by the norm map from $K^{\dagger}$ to $K$.
The following is a commutative
diagram.

\[
\begin{array}{ccccc}
{\rm Div}^{0}(C)(K^{\dagger}) & \rightarrow &
J(K^{\dagger}) & \rightarrow & 0 \\
\downarrow N & & \downarrow N & & \;\;\;\downarrow N \\
{\rm Div}^{0}(C)(K) & \rightarrow & J(K) & \rightarrow &
{\rm coker}\end{array}\]
The composition of the natural inclusion of ${\rm Div}^{0}(C)(K)$ in
${\rm Div}^{0}(C)(K^{\dagger})$
with the norm map to ${\rm Div}^{0}(C)(K)$ is
the multiplication by $d^{\dagger}$ map. So $d^{\dagger}$ kills the cokernel.
The cokernel is thus killed by $d$, the greatest common divisor of the
$d^{\dagger}$'s.
Since $q$ is relatively prime to $d$,
we know that every element of $J(K)/\phi A(K)$
is represented by a divisor class that contains a divisor of $C$
defined over $K$.

Fix a prime $\Gs\in S$.
Since $p | \!\! / d$, there is a completion of $K^{\dagger}$
at a prime over $\Gs$ whose degree over $K_{\Gs}$ is prime to $p$.
The same argument as above shows that every element of $J(K_{\Gs})/
\phi A(K_{\Gs})$ is represented by a divisor class that contains a divisor
of $C$ defined over $K_{\Gs}$.
\end{proof}

Now let us consider Assumption II.
If the induced map $w$ on cohomology groups is an injection for
some given spanning set then we have an injection for any spanning
set containing the given one.
To test whether there exists a spanning set satisfying Assumption II, it
suffices to use the entire set $\Psi$. For any given spanning set,
the following provides a simplification.
Let $K'$ be the minimal field of definition of the $D_{i}'s$.
Let $L_{j}' = \prod \overline{K}_{i_{j}}$ where
$i_{j}$ ranges over all those $l$ such that $D_{l}$ is in the same
$\Gal (K'/K)$-orbit
as $D_{j}$. We let $L_{j}'$ inherit its $\Gal (K'/K)$-module
structure from $L'$.
We have
$L'\cong\prod_{j\in \Lambda}
L_{j}'$.
The group $H^{1}(\Gal (K'/K),\mu_{q}(L'))$ is isomorphic to
$\oplus_{j\in\Lambda} H^{1}(\Gal (K'/K),
\mu_{q}(L_{j}'))$.
Let ${\rm Stab}_{j}$ be the stabilizer in $\Gal (K'/K)$ of $D_{j}$.
The latter sum
is isomorphic to $\oplus_{j\in\Lambda} H^{1} ({\rm Stab}_{j},\mu_{q})$ from
Shapiro's lemma
(see \cite[p.\ 99]{AW}). We can make an analogous statement by replacing
$K$ with $K_{\Gs}$.

\subsection{Extension of the algorithm}
\label{ext}

Let $f$ be a polynomial of degree 2d over the number field $K$ with
distinct roots in $\overline{K}$.
Let $C$ be the normalization of the curve given by the affine equation
$y^{2}=f(x)$ and let $J$ be its Jacobian.
Let $\{ \alpha_{i} \} $ be the set of roots of $f$.
Then the divisor classes in $P_{L}=
([(\alpha_{1},0)-(\alpha_{1},0)],
[(\alpha_{2},0) - (\alpha_{1},0)], \ldots , [(\alpha_{2d},0) -
(\alpha_{1},0)])$ span $J[2]$.
We can
let $L=K[T]/(f(T))$, and $L'=\overline{K}[T]/
(f(T))$ where $\Gal (\overline{K}/K)$ acts trivially on $T$.
Note $\overline{K}[T]/(f(T))$ $\cong \prod \overline{K}[T]/(T-\alpha_{i})$
$\cong \prod \overline{K_{i}}$
by $T\mapsto (T,\ldots ,T) \mapsto (\alpha_{1},\ldots ,\alpha_{2d})$.
The map $w'$, given by $P\mapsto e_{2}(P,\hat{P}_{L})$,
sends $J[2]$ to $\mu_{2}(L')$. However this
map may not be defined over $K$. The group $G$
generated by the set
$\{ (\sigma - 1)(w'(P))\; | \; P\in J[2], \sigma \in \Gal (\overline{K}
/K) \} $ is contained in $\pm 1\subset \mu_{2}(L')$.
Thus the induced map from $J[2]$ to $\mu_{2}(L')/\pm 1$
is defined over $K$; it is also injective. The map this induces on cohomology
may not be injective, but the kernel is sufficiently under control so
that Mordell-Weil ranks can nevertheless sometimes be computed.
In this case, the map $F$ is $x-T$ and its image
is in $L^{\ast}/L^{\ast 2}K^{\ast}$.
This case is far more complicated than
those that fit into the framework given earlier.
It is discussed in
\cite{Ca}, \cite{FPS}, \cite{PS} and \cite{St1}.

We can try to use this technique in general. We can pick some
spanning set and quotient out by a group like $G$.
But then we can not typically expect the
induced map on cohomology to be close enough to injective to
be useful.

\section{Curves of the form $y^{p}=f(x)$}
\label{ytop}

Let $K$ be a field of characteristic 0. Let $f(x)$ be a monic
polynomial over $K$ of degree $d$ with distinct roots in $\overline{K}$.
Let $C$ be the normalization of
the projective curve defined by the affine equation $y^{p}=f(x)$,
where $p$ is a prime that does not divide $d$. (The case where $p$ does
divide $d$ is described in \cite{PS}).
{}From \cite[\S 1]{Tw}, the genus of $C$ is $g=(p-1)(d-1)/2$.
Since $p\neq d$ there is a single point on the line at infinity.
If $d\neq p\pm 1$, then the projective curve will
be singular at $\infty$.
Also from \cite[\S 1]{Tw}, since $p\!\,\not | d$, the normalization
has a single rational point over the point on the line at infinity
which we denote $\infty$.
Since we chose $f$ to have distinct roots, the projective curve given by
$y^{p}=f(x)$
can be singular nowhere else.

Consider the map $\tau$ on $C$, that on the affine part sends
$(x,y)\mapsto (x,\zeta_{p} y)$. The map $\tau$ induces an automorphism
of $J$.
The group $J$ is generated
by divisor classes of the form $[P-\infty]$ where $P$ is an affine
point on $C$. The divisor of the function $x-x(P)$ is
$\tau^{p-1}P + \ldots + \tau P + P - p\infty$.
Consider the subring ${\bf Z}[\tau]$ in
${\rm End}(J)$.
The minimal polynomial of
$\tau$ over ${\bf Z}$ is $t^{p-1}+ \ldots + t+1$. Thus $\tau$ acts
as a primitive $p$th root of unity in ${\rm End}(J)$; so by abuse of
notation we rename it $\zeta_{p}$. Let $\phi=1-\zeta_{p}$ in
End$(J)$.

For $1\leq i\leq p-2$, the quotient of the numbers $(1-\zeta_{p}^{i})$
and $(1-\zeta_{p})$ is a unit. Thus the subgroup $J[\phi]$ of
$J$ is
fixed by $\Gal (\overline{K}/K)$. The abelian variety $J/J[\phi]$,
however,
will typically not be a Jacobian over $K$ unless $K$ contains $\zeta_{p}$.
Thus it will be difficult to compute the Mordell-Weil rank
of $J(K)$ directly for the reasons presented at the end of Section~\ref{alg2}.
For that reason, we assume that $K$ contains $\zeta_{p}$ so
$\phi$ is a $K$-defined endomorphism.

Here is one case where the quotient is a Jacobian.
Let $C$ be $y^{3}=x^{2}-k$ and $C'$ be $y^{3}=x^{2}+27k$ with
$k\in K^{\ast}$ (and $K$ not necessarily containing $\zeta_p$)
and let $E$ and $E'$ be their Jacobians (elliptic curves).
Let $\phi = 1-\zeta_{3}$ on $E$ and $\phi' = 1-\zeta_{3}$ on $E'$.
Then there are isogenies $\tau : E\rightarrow E/E[\phi]\cong E'$
and $\tau' : E'\rightarrow E'/E'[\phi']\cong E$ defined over $K$ with
$\tau'\circ\tau =3$.
In \cite{Tp}, Top describes the computation of the $\tau$- and $\tau'$-Selmer
groups along the lines of Section~\ref{alg}.

Let us show that $\Psi=\lambda^{-1}\hat{J}[\hat{\phi}]=J[\phi]$.
\begin{prop}
\label{pp}
Let $\lambda$ be the canonical principal polarization from $J$ to $\hat{J}$
with respect to $C$.
We have $\lambda^{-1}\hat{J}[\hat{\phi}]=J[\phi]$.
\end{prop}
\begin{proof}
Let $(1-\zeta_{p})^{\dagger}$ denote the image in ${\rm End}(J)$ of
$1-\zeta_{p}$ under the Rosati involution. By definition,
the following diagram commutes.
\[
\begin{array}{rcl}
J & \stackrel{\lambda}{\rightarrow} & \hat{J} \\
(1-\zeta_{p})^{\dagger} \uparrow & & \uparrow \widehat{1-\zeta_{p}} \\
J & \stackrel{\lambda}{\rightarrow} & \hat{J} \end{array}\]
{}From \cite[p.\ 139]{Mi2}, we have $\zeta_{p}^{\dagger}=\zeta_{p}^{-1}$.
Thus we have\[
\Psi = \lambda^{-1}\hat{J}[\widehat{1-\zeta_{p}}]= J[(1-\zeta_{p})^{\dagger}]=
J[1-\zeta_{p}^{\dagger}]=J[1-\zeta_{p}^{-1}]=J[1-\zeta_{p}]=J[\phi]\]
since the quotient of $1-\zeta_{p}^{-1}$ and $1-\zeta_{p}$ is a unit.
\end{proof}

We know $(\phi)^{p-1}=u\cdot p$
where $u$ is a unit in ${\rm End}(J)$.

\vspace{1mm}
\noindent
{\bf Definition:}
Let ${\rm dim}\, M$ denote the dimension of an ${\bf F}_{p}$-vector
space $M$.

\noindent
We also know ${\rm dim}\, J[p]=2g=
(p-1)(d-1)$, so ${\rm dim}\, J[\phi]=d-1$.
We need to choose a
Galois-invariant
spanning set of $J[\phi]$.
Let $\{ \alpha_{i}\}$ be the set of roots of $f$.

\begin{prop}
\label{basis}
The divisor classes
$[(\alpha_{1},0)-\infty],\ldots ,[(\alpha_{d-1},0)-\infty]$ form
a basis for $J[\phi]$.
\end{prop}

\begin{proof}
The following proof was suggested by Michael Stoll.
Let $\overline{K}(C)$ be the function field of $C$ over $\overline{K}$
and let Princ denote the principal divisors.
The following sequences are both exact. \[
0\rightarrow \overline{K}^{\ast}\rightarrow \overline{K}(C)^{\ast}
\stackrel{\rm div}{\rightarrow} {\rm Princ}
\rightarrow 0\]
\[
0\rightarrow {\rm Princ} \rightarrow {\rm Div}^{0}(C)
\rightarrow
J\rightarrow 0\]
Let $\tau$, as before, be the automorphism of $C$ given on the
affine part by $\tau (x,y)=(x,\zeta_{p}y)$. By extending $\tau$ from
points to divisors, the map
$\tau$ induces
maps on Princ, ${\rm Div}^{0}(C)$ and $J$.
We can let $\tau$ act on $\overline{K}(C)$ by fixing $\overline{K}(x)$
and sending
$y$ to $\zeta_{p}^{-1}y$. Let $G=\langle \tau \rangle$.
Under these actions, both are exact sequences
of ${\bf Z}[G]$-modules.

Under $G$-cohomology we have the following exact sequence.
\[
0 = H^{1}(G,\overline{K}(C)^{\ast})\rightarrow H^{1}(G,{\rm Princ})
\rightarrow H^{2}(G,\overline{K}^{\ast})=\overline{K}^{\ast}/
\overline{K}^{\ast p} = 1\]
To get the first equality, we can identify $G$ with $\Gal (\overline{K}(C)/
\overline{K}(x))$ and use Hilbert's theorem 90. The next-to-last
equality comes from the fact that $G$ is a finite cyclic group and
so $H^{2}(G,\overline{K}^{\ast})\cong
{\rm ker}(1-\tau)/{\rm image}({\rm Norm})$, where ${\rm Norm}=1+\ldots
+ \tau^{p-1}$. Therefore
$H^{1}(G,{\rm Princ})=0$ and hence the map from ${\rm Div}^{0}(C)^{G}$ to
$J^{G}=J[\phi]$ is surjective. So $J[\phi]$ is generated by
$G$-invariant divisors. The group ${\rm Div}^{0}(C)^{G}$ is generated
by divisors of the form $P-\infty$ where $P\in C^{G}$ and by
those of the form
${\rm Norm}(P-\infty)$ for arbitrary $P\in C \setminus\infty$.
Each such ${\rm Norm}(P-\infty)
={\rm div}(x-x(P))$ and so is principal. Thus $J[\phi]$ is generated
by divisors of the form $P-\infty$ where $P\in C^{G}$ but the only
points fixed by $G$ are those with $y$-coordinate 0 and $\infty$.

We have already seen that ${\rm dim}J[\phi]=d-1$ and that the
sum of all $d$ divisor classes $[(\alpha_{i},0)-\infty]$ is 0, so
the result follows.
\end{proof}

Therefore $P_{L}=((\alpha_{1},0)-\infty,
\ldots , (\alpha_{d},0)-\infty)$ is a Galois-invariant set whose divisor
classes span
$J[\phi]$. Thus we can set
$L=K[T]/(f(T))$ and
\[
L'=\overline{K}[T]/(f(T))\cong
\prod_{i=1}^{d}
\overline{K}[T]/(T-\alpha_{i})
\cong
\prod_{i=1}^{d} \overline{K}_{i}
\; {\rm  by}\;
T\mapsto (T,\ldots ,T)\mapsto
(\alpha_{1},\ldots ,\alpha_{d})\]
where $\Gal (\overline{K}/K)$
acts trivially on $T$.
Therefore we can let $F = x-T$ where if $R=\sum n_i R_i$
is a good divisor,
then
\[ (x-T)(R) = \prod(x(R_i) - T)^{n_i}\in L^{\ast}.\]
When composed with the isomorphism of $L'$ and $\prod \overline{K_{i}}$,
the map $x-T$ becomes
the $d$-tuple of functions $(x-\alpha_{1},\ldots ,x-\alpha_{d})$, whose
divisors are $pP_{L}$.

It is often convenient to work with divisors of the form $D-r\infty$ where
$D$ is a $K$-rational, effective divisor of degree $r$. Let us consider
the image of the divisor class containing such a divisor under the map
$x-T$. The following proposition holds even when $K$ does not contain
$\zeta_{p}$.

\begin{prop}
\label{transl}
Any element of $J(K)$ can be represented by a divisor of degree 0 which is
defined over $K$ and whose support does not include $\infty$ or points
with $y$-coordinate 0.
In particular, let $D=\ssigmaone Q+\ldots +
\ssigmar Q-r\infty$ where the $\ssigmai\, Q$ are the $r$ conjugates
over $K$ of the point $Q$ of $C$ and $y(Q)\neq 0$. We have
\[
(x-T)([D])\equiv \prod_{i=1}^{r}(x(\ssigmai Q)-T)({\rm mod}\, L^{\ast p}).
\]
Let $D=(\alpha_{1},0)+\ldots + (\alpha_{r},0)-r\infty$
where the $\alpha_{i}$ are conjugate over $K$, possibly renumbered, and $r<d$.
We have
\[(x-T)([D])\equiv
\prod_{i=r+1}^{d}(\alpha_{i}-T)^{-1}+\prod_{i=1}^{r}(\alpha_{i}-T)
\;\;\; ({\rm mod}\, L^{\ast p}).\]
\end{prop}
\begin{proof}
Assume $p$ is odd.
Since $C$ has a $K$-rational point, namely $\infty$, the first sentence
follows from Proposition~\ref{assi}.
Let $Q=(x_{0},y_{0})$ be a point of $C$ defined over a finite extension
of $K$
and fix a set $\{ \sigma_{i} \} $ of embeddings of $K(Q)$ in $\overline{K}$,
such that
$\ssigmaone Q,\ldots ,\ssigmar Q$ are the conjugates of $Q$
over $K$. Let
\[
D=(\sum_{i=1}^{r} \ssigmai Q)-r\infty.\]
Any degree 0 divisor defined over $K$ can be written as the sum and
difference of such divisors (possibly with different sets of
$\{ \sigma_{i}\} $).

First we assume that $y_{0}\neq 0$.
Let $a,b\in K^{\ast}$ with $f(a)\neq 0$.
Let $(a,c)$, $(a,\zeta_{p} c)$, \ldots
$,(a,\zeta_{p}^{p-1} c)$ be the $p$ affine points on the line $x=a$ and let
$(g_{1},b),\ldots ,(g_{d},b)$ be the $d$ affine points on the line $y=b$; the
latter $d$ points are not necessarily distinct.
Since $p$ does not divide $d$ we
can find integers $n$ and $m$ such that $nd+m(d-p)=1$.
The divisor of the function $(x-a)^{mr}(y-b)^{-nr-mr}$ is
\[ r\infty +
rm(a,c) + \ldots + rm(a,\zeta_{p}^{p-1}c) -
r(n+m)(g_{1},b) -\ldots - r(n+m)(g_{d},b).\]

When we add $D$ to this principal divisor we get a divisor without
$\infty$ or points with $y$-coordinate 0
in its support. Therefore we have\[ (x-T)([D])=
\frac{(a-T)^{prm}(x(\ssigmaone Q)-T)\cdot\ldots\cdot (x(\ssigmar Q)-T)}
{((g_{1}-T)\cdot\ldots\cdot (g_{d}-T))^{r(n+m)}}=\]\[
\frac{(a-T)^{prm}(x(\ssigmaone Q)-T)\cdot\ldots\cdot (x(\ssigmar Q)-T)}
{((-1)^{d}(f(T)-b^{p}))^{r(n+m)}}\equiv
(x(\ssigmaone Q)-T)\cdot\ldots\cdot (x(\ssigmar Q)-T) ({\rm mod}\,
L^{\ast p}).\]

Let $D=(\alpha_{1},0)+\ldots + (\alpha_{r},0)-r\infty$
where the $\alpha_{i}$ are conjugate over $K$ and $r<d$.
We can let $r< d$ since $\sum (\alpha_{i},0)-d\infty$ is principal.
Let $g=y-\prod_{i=1}^{r}(x-\alpha_{i})$. We have \[
(g)=(\alpha_{1},0)+\ldots +(\alpha_{r},0)+\sum_{j=1}^{m'}P_{j} -
(m'+r)\infty\]
where $m'={\rm max}\{ d-r,r(p-1) \}$ and
$P_{j}=(x_{j},y_{j})$ with $y_{j}\neq 0$ and the $x_{j}$'s are the
roots of the polynomial\[
\prod_{i=r+1}^{d}(x-\alpha_{i})-\prod_{i=1}^{r}(x-\alpha_{i})^{p-1}.\]
We have\[
D-(g)=m'\infty-\sum_{i=1}^{m'}P_{i}\] which is the negative of a divisor
of the form we handled in the first part of the theorem.
Therefore we have
\[
(x-T)([D])\equiv (\prod_{i=1}^{m'}(x_{i}-T))^{-1}\equiv
(\prod_{i=1}^{m'}(T-x_{i}))^{-1}
\]\[ \equiv
(\prod_{i=r+1}^{d}(T-\alpha_{i})-\prod_{i=1}^{r}(T-\alpha_{i})^{p-1})^{-1}
({\rm mod}\, L^{\ast p}).\]

Let us consider $L$ to be a product of number fields or to be contained
in $\prod \overline{K_{i}}$. In either case, one of the two products
in the above formula will
be 0 at each factor. Thus we have \[
(x-T)([D])\equiv
\prod_{i=r+1}^{d}(\alpha_{i}-T)^{-1}+\prod_{i=1}^{r}(\alpha_{i}-T)
\;\;\; ({\rm mod}\, L^{\ast p}).\]

For the $p=2$ case, see \cite[Lemma 2.2]{Sc1}.
\end{proof}
The upshot of the first formula in the above lemma
is that you can basically ignore the appearance of
$\infty$ in such a divisor.

The following proposition shows that the map $w$ induces on cohomology
is injective and describes $H^{1}(K,J[\phi];S)$.

\begin{prop}
\label{split}
Let $K$ be a number field containing $\zeta_{p}$.
The groups $H^{1}(K,J[\phi];S)$ and \linebreak
${\rm ker}:\;
L(S,p)\stackrel{\rm norm}{\rightarrow}
K^{\ast}/K^{\ast p}$ are isomorphic via $k\circ w$.
\end{prop}
\begin{proof}
First we show that the following is
a split exact sequence of $\Gal (\overline{K}/K)$-modules
\[
0\rightarrow J[\phi]\stackrel{w}{\rightarrow} \mu_{p}(L')
\stackrel{N}{\rightarrow} \mu_{p}(\overline{K})\rightarrow 0\]
where $N$ is the norm map.
The dimensions
of the three ${\bf F}_{p}$-vector spaces are $d-1$, $d$ and 1 respectively.
The divisor of the function $y$ is $(\alpha_{1},0)+\ldots + (\alpha_{d},0)
-d\infty$.
So the sum of the $d$ divisor classes $[(\alpha_{i},0)-\infty]$ is
trivial. Then since $\zeta_{p}\in K$, the Weil pairing is linear and
the image of $w$ is therefore equal to the kernel of the norm.
Let $\Delta$ be the diagonal embedding of $\mu_{p}(\overline{K})$
in $\mu_{p}(L')$. Let $b$ be a positive residue of
$d^{-1}({\rm mod}\; p)$. Then the composition of $\Delta^{b}$ and
the norm is the identity, so the exact sequence splits.

Since this short exact sequence splits, the following is a split
exact sequence
\[
0\rightarrow H^{1}(K,J[\phi])\stackrel{w}{\rightarrow}
H^{1}(K,\mu_{p}(L'))
\stackrel{N}{\rightarrow} H^{1}(K,\mu_{p}(\overline{K}))\rightarrow 0.\]
The group $H^{1}(K,\mu_{p}(L'))$
is isomorphic to $L^{\ast}/L^{\ast p}$ by the map we call $k$.
The group $H^{1}(K,\mu_{p}(\overline{K}))$ is isomorphic to $K^{\ast}/
K^{\ast p}$ by a Kummer map also.
So $k\circ w$ induces an isomorphism of
$H^{1}(K,J[\phi])$ with the kernel of the norm from
$L^{\ast}/L^{\ast p}$ to $K^{\ast}/K^{\ast p}$ and of
$H^{1}(K,J[\phi];S)$ with the kernel of the norm from
$L(S,p)$ to $K^{\ast}/K^{\ast p}$.
\end{proof}

The following proposition has two corollaries for a number field $K$.
The first gives the size of $J(K_{\Gs})/\phi J(K_{\Gs})$
for $\Gs$ a finite prime. The second
shows how to
find the Mordell-Weil rank of $J(K)$ from $J(K)/\phi J(K)$ and knowledge
of torsion.

\begin{prop}
\label{psize}
Let ${\cal K}$ be a number field or the completion of a number field at
a finite prime, that contains $\zeta_{p}$.
Then ${\rm dim} J({\cal K})/pJ({\cal K}) - {\rm dim} J({\cal K})[p]$
is the same as
$(p-1)({\rm dim} J({\cal K})/\phi J({\cal K}) - {\rm dim} J({\cal K})[\phi])$.
\end{prop}
\begin{proof}
For each $n\geq 1$, the following is an exact sequence from
Proposition~\ref{snake} where $B=D=J({\cal K})$ and $g=\phi$ and
$f=\phi^{n}$.
\[
0\rightarrow \frac{J({\cal K})[\phi^{n}]}{\phi
(J({\cal K})[\phi^{n+1}])}\rightarrow
\frac{J({\cal K})}{\phi J({\cal K})}\stackrel{\phi^{n}}{\rightarrow}
\frac{J({\cal K})}{\phi^{n+1}J({\cal K})}\rightarrow
\frac{J({\cal K})}{\phi^{n}J({\cal K})}
\rightarrow 0\]
Therefore \[
{\rm dim}J({\cal K})/pJ({\cal K}) =
{\rm dim}J({\cal K})/\phi^{p-1}J({\cal K}) \]
\[
= (p-1){\rm dim}J({\cal K})/\phi J({\cal K}) - \sum_{i=1}^{p-2}
{\rm dim}J({\cal K})[\phi^{i}] + \sum_{i=1}^{p-2}{\rm dim}\phi
(J({\cal K})[\phi^{i+1}]).\]
Thus
\[
{\rm dim}J({\cal K})/pJ({\cal K}) - {\rm dim}J({\cal K})[p] =
(p-1){\rm dim}J({\cal K})/\phi J({\cal K}) - \sum_{i=1}^{p-1}
{\rm dim}J({\cal K})[\phi^{i}] + \sum_{j=1}^{p-1}{\rm dim}\phi
(J({\cal K})[\phi^{j}]).\]
For each $j\geq 1$, the following is an exact sequence.
\[
0\rightarrow J({\cal K})[\phi] \rightarrow
J({\cal K})[\phi^{j}]\stackrel{\phi}
{\rightarrow}\phi J({\cal K})[\phi^{j}]\rightarrow 0\]
Thus\[
{\rm dim}J({\cal K})/pJ({\cal K}) - {\rm dim}J({\cal K})[p] =
(p-1){\rm dim}J({\cal K})/\phi J({\cal K}) -
(p-1){\rm dim}J({\cal K})[\phi].\]
\end{proof}

\begin{cor}
\label{localsize}
Let $K_{\Gs}$ be a finite extension of ${\bf Q}_{s}$ containing
$\zeta_{p}$ and let
$r={\rm ord}_{p}(s)$. In addition let $g$ be the
genus of $C$. Then
${\rm dim}J(K_{\Gs})/\phi J(K_{\Gs}) = gr[K_{\Gs}:{\bf Q}_{s}(\zeta_{p})] +
{\rm dim}J(K_{\Gs})[\phi]$.
\end{cor}

\begin{proof}
If $s\neq p$, then $r=0$ and this follows from Proposition~\ref{size}.
Let $\Gs$ lie over $p$; so $r=1$. From Proposition~\ref{size}, we have
\[
{\rm dim}J(K_{\Gs})/pJ(K_{\Gs})=
g[K_{\Gs}:{\bf Q}_{p}] + {\rm dim} J(K_{\Gs})[p].\]
Using Proposition~\ref{psize} we have
\[
g[K_{\Gs}:{\bf Q}_{p}]=
(p-1)({\rm dim} J(K_{\Gs})/\phi J(K_{\Gs}) - {\rm dim} J(K_{\Gs})[\phi])\]
\[ g[K_{\Gs}:{\bf Q}_{p}(\zeta_{p})]=
{\rm dim}J(K_{\Gs})/\phi J(K_{\Gs}) - {\rm dim}J(K_{\Gs})[\phi].\]
\end{proof}

\begin{cor}
\label{pMWrk}
Let $K$ be a number field containing $\zeta_{p}$. The Mordell-Weil rank
of $J(K)$ is \linebreak
$(p-1)({\rm dim} J(K)/\phi J(K) - {\rm dim} J(K)[\phi])$.
\end{cor}
This follows immediately from Proposition~\ref{psize}.

We conclude with a proposition suggested independently by
Armand Brumer, Michael Stoll and the referee. The proof appears
after \cite[Lemma 13.4]{PS}.

\begin{prop}
\label{Stoll}
Let $C$ be defined over $K$, a number field
not necessarily containing $\zeta_{p}$.
The Mordell-Weil rank of $J(K)$ is the quotient of the
Mordell-Weil rank of $J(K(\zeta_{p}))$ by $[K(\zeta_{p}):K]$.
\end{prop}

\subsection{Example where not all elements of $\Psi$ are rational}

\begin{prop}
\label{twelve}
Let $C$ be the projective curve given by the affine
equation $y^{3}=(x^{2}+1)
(x^{2}-4x+1)$
and let $J$ be its Jacobian. The group $J({\bf Q})$ has
Mordell-Weil rank 1 and the group $J({\bf Q}(\zeta_{3}))$ has Mordell-Weil
rank 2.
\end{prop}

\begin{proof}
Let $K={\bf Q}(\zeta_{3})$ and $f(x)=(x^{2}+1)
(x^{2}-4x+1)$. Let $\phi = 1-\zeta_{3}$.
We first compute
$J(K)/\phi J(K)$.
The roots of $f$ are $\pm i$ and $2\pm \sqrt{3}$.
We have $L=K[T]/(f(T))\cong
K(i)\times K(i)$ by $T\mapsto (i,2+\sqrt{3})$ and $L^{\ast}/L^{\ast 3}$
is isomorphic to $(K(i)^{\ast}/K(i)^{\ast 3})^{2}$.
In $K(i)={\bf Q}(\zeta_{12})$, we fix
$\zeta_{3} = (-1+\sqrt{-3})/2$ and
$\sqrt{3}=i\sqrt{-3}$.
The bad primes of $C$ over ${\bf Q}$ are 2 and 3.
There is one prime of $K(i)$ over 2 generated by $(1+i)$; it is inert in $K$
and ramifies in $K(i)$.  There is one prime $\Gq$ of $K(i)$ over 3; it
ramifies in $K$ and then is inert up to $K(i)$.
We will denote the restriction of these primes to $K$ by 2 and $\Gq$.
Since $K(i)$ is a totally
imaginary extension of the rationals, it has unit rank 1.
We note that $i-\zeta_{3}$ is a
fundamental unit.
The class
group of the field $K(i)$ is trivial.  Thus $K(i)(S,3)$ is $\langle
i-\zeta_{3}, \zeta_{3}, 1+i, \sqrt{-3}\rangle$.

{}From Proposition~\ref{split},
the group $H^{1}(K,J[\phi];S)$ is the kernel of the norm from
$L(S,3)\cong K(i)(S,3)^{2}$ to $K^{\ast}/K^{\ast 3}$.
The number $\zeta_{3}(i-\zeta_{3})$
generates the kernel
of the norm from $K(i)(S,3)$ to
$K^{\ast}/K^{\ast 3}$. Thus $H^{1}(K,J[\phi];S)=\langle (i-\zeta_{3},
(i-\zeta_{3})^{2}), (\zeta_{3},\zeta_{3}^{2}),(1+i,(1+i)^{2}),
(\sqrt{-3},\sqrt{-3}^{2}), (1,\zeta_{3}(i-\zeta_{3}))\rangle$.
The group $S^{\phi}(K,J)$
is the intersection of the groups $\beta_{\Gs}^{-1}(F_{\Gs}(J(K_{\Gs})
/\phi J(K_{\Gs})))$ for
the primes $\Gs=2$ and $\Gq$ of $K$.

At this point let us find the images of the known elements of $J(K)$
by the map $x-T$.
The group $J(K)[\phi]$ has order 3 and is generated by the divisor
class $[(i,0)+(-i,0)-2\infty]$.
In $J(K)$ there is also the divisor
class $[(0,1)-\infty]$. In the following table we present
the images of these two classes in $H^{1}(K,J[\phi];S)$ by the map $x-T$.
Above each coordinate is written $x-\alpha$ to remind
us how to compute that coordinate. We use Proposition~\ref{transl}
to compute
the images of $[(0,1)-\infty]$ and
$[(i,0)+(-i,0)-2\infty]$.
\[
\begin{array}{rrccc}
 & & & x-i & x-(2+\sqrt{3}) \\
 & [(i,0)+(-i,0)-2\infty]
 & \mapsto & (1+i)^{2} & \zeta_{3}^{2}(i-\zeta_{3})^{2}(1+i) \\
 &[(0,1)-\infty] & \mapsto & 1 & \zeta_{3}^{2}(i-\zeta_{3})^{2} \end{array}\]
{}From the first exact sequence in the proof of Proposition~\ref{psize},
we know $J(K)[\phi]/\phi (J(K)[3])$ injects into $J(K)/\phi J(K)$ which
injects into $L^{\ast}/L^{\ast 3}$. Thus we know that
$J(K)[3]=J(K)[\phi]$ and is generated by $[(i,0)+(-i,0)-2\infty]$
since its image is not trivial.
In addition we see that the image of $[(0,1)-\infty]$ is independent of
the image of torsion and so the divisor class
has infinite order. We will show that
$S^{\phi}(K,J)$ is generated by the images of
$[(i,0)+(-i,0)-2\infty]$ and
$[(0,1)-\infty]$.

Let us describe the groups
$J(K_{\Gq})/\phi
J(K_{\Gq})$ and $L_{\Gq}^{\ast}/L_{\Gq}^{\ast 3}\cong
(K_{\Gq}(i)^{\ast}/K_{\Gq}(i)^{\ast 3})^{2}$.
The group $K_{\Gq}(i)^{\ast}/K_{\Gq}(i)^{\ast 3}$ is
$\langle \sqrt{-3}, 1+\sqrt{-3}, 1+i\sqrt{-3}, 1+\sqrt{-3}^{2}$,
$1+i\sqrt{-3}^{2}$, $1+\sqrt{-3}^{3}\rangle$. Let us rename those numbers
$\langle A, B, E, \Gamma, \Phi, \Delta\rangle$, to agree with the
notation in \cite{KS}.
Multiplicatively, anything that is
1 modulo 9 is a cube, as are $1\pm i\sqrt{-3}^{3}$ and $-1$.
We have $[i-\zeta_{3},\zeta_{3},1+i,\sqrt{-3}]\equiv [BE,B^{2}\Gamma,
\Gamma^{2},A]\, {\rm mod}\, K_{\Gq}(i)^{\ast 3}$.
Thus the kernel of $\beta_{\Gq}$ is trivial.
{}From Corollary~\ref{localsize}
we know
${\rm dim}\, J(K_{\Gq})/\phi J(K_{\Gq})$
is the sum of $3[K_{\Gq}:{\bf Q}_{3}(\zeta_{3})]$, which is 3, and
${\rm dim}\, J(K_{\Gq})[\phi]$, which is 1
since $K_{\Gq}\cap K(J[\phi])=K$, for a total of 4.
In the following table we list generators of
$J(K_{\Gq})/\phi
J(K_{\Gq})$
and their images in $L_{\Gq}^{\ast}/L_{\Gq}^{\ast 3}\cong
(K_{\Gq}(i)^{\ast}/K_{\Gq}(i)^{\ast 3})^{2}$.
\[
\begin{array}{rrccc}
 &  & & x-i & x-(2+\sqrt{3}) \\
 & [(i,0)+(-i,0)-2\infty] & \mapsto & \Gamma & \Gamma E^{2} \\
 & [(0,1)-\infty] & \mapsto & 1 & \Gamma^{2}E^{2} \\
 & [(4,y_{1})-\infty] & \mapsto & \Phi & \Gamma E \\
 & [(\frac{1+\sqrt{-3}^{3}}{\sqrt{-3}^{3}},y_{2})-\infty]
 & \mapsto & \Delta & \Delta^{2}
\end{array}\]

A small amount of linear algebra shows that
$\beta_{\Gq}^{-1}(F_{\Gq}(J(K_{\Gq})/\phi J(K_{\Gq})))$
is the same as the group generated by the images of $[(i,0)+(-i,0)-2\infty]$
and $[(0,1)-\infty]$. So that is the Selmer group and those two
divisor classes generate $J(K)/\phi J(K)$.
We do not even need $J(K_{2})/\phi J(K_{2})$.
Thus, from
Corollary~\ref{pMWrk}, the Mordell-Weil rank of $J(K)$ is 2. One can
verify that
the divisor classes $[(0,1)-\infty]$ and $[(0,\zeta_{3})-\infty]$ have
infinite order and are
independent.
From Proposition~\ref{Stoll}, the Mordell-Weil rank of $J({\bf Q})$ is 1.
\end{proof}

Using a straightforward computation in the number field gotten by
adjoining to ${\bf Q}$ the root of the characteristic polynomial of
Frobenius of $J$ over ${\bf F}_{7}$,
Michael Stoll has shown that $J$
is absolutely simple. This type of argument appears in the proof of
\cite[Prop.\ 14.4]{PS}.

\subsection{Examples with Mordell-Weil rank 0}

In this section we find solutions of two diophantine equations over
infinitely many number fields.
First let us state two propositions. Each follows from
the Riemann-Roch theorem and results in  \cite[\S 5]{Mi1} and is
well-known.

\begin{prop}
\label{genus2}
Let $f(x)$ be a polynomial of degree 5 or 6,
defined over a field $K$ of characteristic other than 2 with
distinct roots in $\overline{K}$.
Let $C$ be the normalization of the curve whose
affine equation is $y^{2}=f(x)$. Every element of ${\rm Pic}^{2}(C)$
has a unique representation by an effective divisor, with the exception
of the canonical class. In addition, every $K$-rational divisor class
of degree 2 can
be represented by an effective $K$-rational divisor.
\end{prop}

\begin{prop}
\label{genus3}
Let $C$ be a smooth plane quartic curve defined over a field $K$.
Every element of ${\rm Pic}^{3}(C)$
has a unique representation by an effective divisor unless the divisor class
contains $P_{1}+P_{2}+P_{3}$ where the three $P_{i}$'s are collinear.
In the latter case $[P_{1}+P_{2}+P_{3}]=[Q_{1}+Q_{2}+Q_{3}]$ if and only
if there are lines $L_{1}$ and $L_{2}$ and a point $R$ such that
$L_{1}.C=P_{1}+P_{2}+P_{3}+R$ and $L_{2}.C=Q_{1}+Q_{2}+Q_{3}+R$.
Assume, in addition, that $C$ has a $K$-rational point.
Every $K$-rational divisor class of degree 3
contains an effective $K$-rational divisor.
\end{prop}

From these follow special cases of the fact that when $C$ has a $K$-rational
point and the group $J(K)$ is finite, then we can describe all points on
$C$ over fields of degree over $K$, less than or equal to the genus.
Though we can describe all such points, it is a more difficult problem to
pick one of the fields and decide which of those points are defined
over that field.
We present examples using each of the previous propositions.

\begin{prop}
\label{five}
The only ${\bf Q}$-rational
points on the curve $C$ given by $y^{2}=x^{5}+1$ are
$\infty$, $(0,\pm 1)$ and $(-1,0)$. The only other points on $C$ rational over
quadratic extensions are $(1+i, \pm (1-2i))$, $(1-i, \pm (1+2i))$
and those with $x\in {\bf Q}$.
\end{prop}
Note that we could compute a 2-Selmer group or a $(1-\zeta_{5})$-Selmer group.
We will do the latter, as there are already examples of the former in
the literature.

\begin{proof}
We can rewrite the curve as $x^{5}=(y+1)(y-1)$ and let $K={\bf Q}(\zeta_{5})$.
We use the endomorphism $\phi=1-\zeta_{5}$ of $J$, the Jacobian of $C$.
The bad primes are the single prime over 2, which we also denote by 2,
and the single
prime $\Gp=1-\zeta_{5}$ over 5. The field $K$ has class number 1 and unit rank
1, with fundamental unit $1+\zeta_{5}$.
Thus $K(S,5)=\langle \zeta_{5}, 1+\zeta_{5}, 2, 1-\zeta_{5}\rangle$
and $H^{1}(K,J[\phi];S)$ is the kernel of the norm from $L(S,5)=K(S,5)^{2}$
to $K^{\ast}/K^{\ast 5}$.

We have $K_{\Gp}^{\ast}/K_{\Gp}^{\ast 5} = \langle \Gp, 1+\Gp,
1+\Gp^{2}, 1+\Gp^{3}, 1+\Gp^{4}, 1+\Gp^{5}\rangle$. Let us rename
those elements of $K_{\Gp}$ by $\langle a, b, c, d, e, f\rangle$.
Any element that is 1 modulo $\Gp^{6}$ is a fifth power, as are the
fourth roots of unity.
The vectors $[\zeta_{5},1+\zeta_{5},2,1-\zeta_{5}]\equiv
[b^{4}c^{4}e^{4},b^{2}c^{4}d^{2}e^{4},e^{3}f,a]\, {\rm mod}\,
K_{\Gp}^{\ast 5}$.
Thus the kernel of $\beta_{\Gp}$ is trivial.
{}From Corollary~\ref{localsize},
$J(K_{\Gp})/\phi J(K_{\Gp})$ has dimension 3.
In the following table we list generators
of $J(K_{\Gp})/\phi J(K_{\Gp})$
and their images in $L_{\Gp}^{\ast}/L_{\Gp}^{\ast 5}$.

\[
\begin{array}{rrcc}
 & & y+1 & y-1 \\
& [(x_{1},\Gp^{3}) - \infty]\mapsto & d & d^{4} \\
& [(x_{2},\Gp^{4}) - \infty]\mapsto & e & e^{4} \\
& [(x_{3},\Gp^{5}) - \infty]\mapsto & f & f^{4}
\end{array}\]
We see that $\beta_{\Gp}^{-1}$ of the image of $J(K_{\Gp})/\phi J(K_{\Gp})$
is the group generated by the image of
$[(0,1)-\infty]$.
So that is the Selmer group
and $J(K)/\phi J(K)$ is generated by that divisor class. We do not even
need $J(K_{2})/\phi J(K_{2})$.
{}From Corollary~\ref{pMWrk},
we see that $J(K)$ has Mordell-Weil rank 0.

Now $\# J({\bf Q})$ is at least 10 since $J({\bf Q})$ contains
$[(0,1)-\infty]$ of order 5 and $[(-1,0)-\infty]$ of order 2.
By computing $\# J({\bf F}_{p})$ for a few primes we can prove that
the order of $J({\bf Q})$ divides 10 so it is equal to 10.
Let $D=[(0,1)+(-1,0)-2\infty]$. Then
$2D=[2(0,1)-2\infty]$, $3D=[(1+i,1-2i)+(1-i,1+2i)-2\infty]$,
$4D=[(0,-1)+\infty-2\infty]$,
$5D=[(-1,0)+\infty-2\infty]$, $6D=[(0,1)+\infty-2\infty]$,
$7D=[(1+i,-1+2i)+(1-i,-1-2i)-2\infty]$,
$8D=[2(0,-1)-2\infty]$, $9D=[(0,-1)+(-1,0)-2\infty]$, $10D=[2\infty-
2\infty]=0$.

% y=x^3+1 intersects at 3(0,1) + (-1,0) + (1+i,-1+2i) + (1-i,-1-2i)

We have a bijection of the sets
$J({\bf Q})$ and ${\rm Pic}^{2}(C)({\bf Q})$
by $[P+Q-2\infty] \mapsto [P+Q]$.
{}From Proposition~\ref{genus2},
if $P$ is a ${\bf Q}$-rational point of $C$, then $[P+\infty-2\infty]$ must
appear in the above list. If $P$ is defined over a quadratic extension
of ${\bf Q}$ and $\overline{P}$ is its conjugate, then $[P+\overline{P}-
2\infty]$
must appear in the above list, unless $[P+\overline{P}-2\infty]$
is the canonical
class. A simple calculation shows that
if that is the case then $x(P)\in {\bf Q}$.
\end{proof} Since $J$ has complex multiplication by a cyclic, quartic, totally
imaginary field it is absolutely simple. See \cite{St2} for more discussion
of the Mordell-Weil ranks of the Jacobians of curves of the form
$y^{2}=x^{l}+k$ where $l$ is an odd prime.

\begin{prop}
\label{consec}
The only ${\bf Q}$-rational points on the curve $C$ given by
$y^{3}=x(x-1)(x-2)(x-3)$ are $\infty$ and those on $y=0$.
The only other points over quadratic extensions of ${\bf Q}$ are those on
$y=-1$
and $y=2$. There are 12 conjugate triples of points over cubic extensions
that are not collinear. All other points over cubic extensions can be
obtained by finding the other three points of intersection of a
${\bf Q}$-rational
line with a point of $C({\bf Q})$.
\end{prop}

\begin{proof}
We can use $\phi=1-\zeta_{3}$ and
the techniques in earlier examples, to show that
$J$ has trivial Mordell-Weil rank over ${\bf Q}(\zeta_{3})$ and hence
over ${\bf Q}$. Let us compute $J({\bf Q})$.
We already have all of $J[\phi]$
rational over ${\bf Q}$. The line $y=-1$ is bitangent to $C$ and
meets the
curve at $2(x_{1},-1)+2(x_{2},-1)$ where the $x_{i}$ are the roots of
$x^{2}-3x+1$.
A line $L$ is a bitangent of
$C$ if the intersection divisor of $L$ with $C$ is
$L.C=2P+2Q$ for points $P$ and $Q$ of $C$ (not necessarily distinct).
We have a bijection of the sets $J({\bf Q})$ and
${\rm Pic}^{3}(C)({\bf Q})$ by $[P+Q+R-3\infty]\mapsto [P+Q+R]$.
{}From Proposition~\ref{genus3}, the order of $[(x_{1},-1)+(x_{2},-1)+\infty
-3\infty]$ is 2.

The primes of bad reduction over ${\bf Q}$ are 2 and 3.
The characteristic polynomial of the Frobenius of $J$ over ${\bf F}_{5}$
is $f_{5}(t)=t^6 -3t^4 - 15t^2 + 125$ which factors over
${\bf Q}$ into irreducible
quadratic and quartic factors. Thus $J$ is isogenous over ${\bf Q}$
to the sum of the elliptic curve $E$ given by
$y^{3}=(\hat{x}-9/4)(\hat{x}-1/4)$ (where $\hat{x} = (x-\frac{3}{2})^{2}$)
and a 2-dimensional  abelian variety which is simple over ${\bf Q}$.
In addition $\# J({\bf F}_{5})=f_{5}(1)=4\cdot 27$.
The divisor of $y+1$ is $4(-1,-1)-4\infty$ over ${\bf F}_{5}$. From
Proposition~\ref{genus3}, the order of $[(-1,-1)+2\infty -3\infty]$ is 4
in $J({\bf F}_{5})$. So the 2-power part
of $J({\bf F}_{5})$ is a cyclic group of order 4.

We have $\# J({\bf F}_{19})=16\cdot 27\cdot 13$. The curve has 10 rational
bitangents over ${\bf F}_{19}$. They are
the line at infinity, $uy=-1$, $uy=4x+10$
and $uy=10x+2$ where $u^{3}=1$. This gives us 9 divisor classes of the form
$[P_{1}+P_{2}+\infty - 3\infty]$ where $2P_{1}+2P_{2}$ is the intersection divisor
of $C$ with one of the ${\bf F}_{19}$-rational bitangents which is
not the line at infinity. Each of these 9 divisor
classes is different and has order 2, from Proposition~\ref{genus3}.
The 2-power part of $J({\bf F}_{19})$ has 16
elements and at least 9 have order 2. Thus the 2-power part has exponent 2.
Putting together the information from the reductions at 5 and 19, we
see that $J({\bf Q})\cong ({\bf Z}/3{\bf Z})^{3}\oplus {\bf Z}/2{\bf Z}$.
By comparison $E({\bf Q})\cong {\bf Z}/3{\bf Z}\oplus {\bf Z}/2{\bf Z}$.

We can find an effective
representative of each divisor class in ${\rm Pic}^{3}(C)({\bf Q})$.
They are supported on the point $\infty$, the four points on $y=0$,
the two points
on the bitangent $y=-1$,
the four points on $y=2$ (each is quadratic over ${\bf Q}$),
and 12 triples of non-collinear conjugate
cubic points.
The proposition then follows from Proposition~\ref{genus3}.

Of course $E({\bf Q})$ gives us the ${\bf Q}$-rational points. But it does
not give the points defined over quadratic and cubic extensions.
\end{proof}

\section{A 2-descent for the Jacobian of
a smooth plane quartic curve using bitangents}
\label{bitan}

Let $C$ be the curve over ${\bf Q}$ defined by the equation
$592900X^{4} + (-1609300Y + 1829520Z)X^{3} + (1253725Y^{2} - 244420ZY +
1648504Z^{2})X^{2} + (-219450Y^{3} - 220390ZY^{2} + 58564Z^{2}Y
+ 365904Z^{3})X
+ (11025Y^{4} + 6510ZY^{3} - 31379Z^{2}Y^{2} - 9548Z^{3}Y + 23716Z^{4})=0$ and
let $J$ be its Jacobian.
The lines $X=0$, $Y=0$, $Z=0$, $X+Y+Z=0$, $X-Y-2Z=0$, $2X-Y+Z=0$, and
$X-3Y+2Z=0$ are all bitangents of $C$
(see the proof of Proposition~\ref{consec} for the definition of bitangent).
The curve $C$ is a smooth plane quartic curve
and so has genus 3.
We will work over $K={\bf Q}$ and
use the multiplication by 2 map from $J$ to itself as our isogeny.
In this case $\Psi=\lambda^{-1}\hat{J}[\hat{2}]=J[2]$, where $\lambda$
is the canonical principal polarization of $J$ with respect to $C$.
The curve $C$
has the property that every element
of $J[2]$ is defined over ${\bf Q}$. This fact
simplifies the example and makes Assumption II hold.

Away from the line $Z=0$ we let
$x=X/Z$ and $y=Y/Z$ and denote points by their affine, $(x,y)$-coordinates.
The divisors of the functions $x$, $y$, $x+y+1$, $y-x+2$, $y-2x-1$ and
$x-3y+2$ are doubles of divisors, all of whose images have order
two in $J$; in fact they form a basis for $J[2]$.
Thus we can let $L={\bf Q}^{6}$ and $F=(x,y,x+y+1,y-x+2,
y-2x-1,x-3y+2)$.
The group $H^{1}({\bf Q},J[2])$ is isomorphic to
$L^{\ast}/L^{\ast 2}\cong
({\bf Q}^{\ast}/{\bf Q}^{\ast 2})^{6}$ by the map $k\circ w$.
The map $F$ is an injection from
$J({\bf Q})/2J({\bf Q})$ to $({\bf Q}^{\ast}/{\bf
Q}^{\ast 2})^{6}$.
The set $C({\bf Q})$ contains $(-7/5,0)$ and $(-1/7,0)$
(coming from the intersection with $y=0$) and
$(1/5, 8/5)$ so Assumption I holds from
Proposition~\ref{assi}.

It is a straightforward
exercise to show that this curve has nonsingular
reduction at all finite primes
greater than 17 and singular reduction at the others;
thus we can let $S= \{ \infty, 2, 3, 5, 7, 11, 13, 17 \} $.
The image of $J({\bf Q})$ under $F$ in
$({\bf Q}^{\ast}/{\bf Q}^{\ast 2})^{6}$ is contained in the image of
$H^{1}({\bf Q},J[2];S)$. Recall from Step 5 that ${\bf Q}(S,2)=
\langle -1, 2, 3, 5, 7, 11, 13, 17\rangle \subset {\bf Q}^{\ast}/
{\bf Q}^{\ast 2}$.
Under the identification of $H^{1}({\bf Q},J[2])$
with $({\bf Q}^{\ast}/{\bf Q}^{\ast 2})^{6}$, the group $H^{1}({\bf Q},
J[2];S)$ gets sent to $L(S,2)={\bf Q}(S,2)^{6}$.

In the following table, we show the images in $L(S,2)$,
under $F$, of the six elements
of $J[2]$ and two other rational divisor classes.
Along the top of the table,
we list the
component functions of $F$. When we write $\frac{1}{2} (x)$, for example,
we mean the divisor whose double is the divisor of $x$.

\[
\begin{array}{rrcccccc}
 & & x & y & x+y+1 & y-x+2 & y-2x-1 & x-3y+2 \\
& [\frac{1}{2} (x)] \mapsto
 & 5 & 3\cdot 7 & 2\cdot 5 \cdot 7 & 2\cdot 7 & -2\cdot 3\cdot 5\cdot 7 & 3
\cdot 7 \\
& [\frac{1}{2} (y)]\mapsto & -3\cdot 7 & 2\cdot 5\cdot 11 & 2\cdot 7 & 2\cdot 3\cdot 5\cdot 7
& -7 & -5 \cdot 7 \\
& [\frac{1}{2} (x+y+1)]\mapsto & -2 \cdot 5 \cdot 7 & -2 \cdot 7 & -2 \cdot 5 & 7 &
 -3 \cdot 5 \cdot 7 & -7 \\
& [\frac{1}{2} (y-x+2)]\mapsto & -2 \cdot 7 & -2 \cdot 3 \cdot 5 \cdot 7 & -7 & 2 \cdot
5 \cdot 17 & -2 \cdot 7 & -3 \cdot 5 \cdot 7 \\
& [\frac{1}{2} (y-2x-1)]\mapsto & 2 \cdot 3 \cdot 5 \cdot 7 & 7 & 3 \cdot 5 \cdot 7 & 2
\cdot 7 & 1 & 3 \cdot 5 \cdot 7 \\
& [\frac{1}{2} (x-3y+2)]\mapsto & -3 \cdot 7 & 5 \cdot 7 & 7 & 3 \cdot 5 \cdot 7 & -3
\cdot 5 \cdot 7 & -13 \\
& [(-7/5,0) - (-1/7,0)]\mapsto & 5 & -7 & -3 \cdot 5 \cdot 7 & 3 \cdot 7 \cdot 17
& -7 & 3 \cdot 5 \cdot 7 \cdot 13 \\
& [(1/5,8/5) - (-1/7,0)] \mapsto
& -5 \cdot 7 & 3 \cdot 11 & 3 \cdot 5 & 3 \cdot 7 \cdot
17 & -7 & -5 \cdot 7 \end{array}\]
The images of all eight divisor classes are independent.
Since the images of the six divisor classes of order two are independent,
they form a basis for $J[2]$. As the images of the other two divisors
are independent of the image of $J[2]$, and of each other in $J({\bf Q})/
2J({\bf Q})$, they
each have infinite order and are independent in the Mordell-Weil group.

We can compute $S^{2}({\bf Q},J)$ in a manner similar to previous examples.
The only difference is that we need to compute the intersection
of all eight groups
$\beta_{s}^{-1}(F_{s}(J({\bf Q}_{s})/2J({\bf Q}_{s})))$ for $s\in S$.
The intersection has dimension 8 as an ${\bf F}_{2}$-vector space and
a basis is the image of the eight rational divisors in the table.
Thus ${\rm dim}_{{\bf F}_{2}}J({\bf Q})/2J({\bf Q})=8$ and
${\rm dim}_{{\bf F}_{2}}J({\bf Q})[2]=6$
and so the Mordell-Weil
rank is exactly 2.

\begin{prop}
\label{bit}
The Mordell-Weil group over ${\bf Q}$ of the Jacobian of the smooth plane
quartic curve $C$, which is
bitangent to $X=0$, $Y=0$, $Z=0$, $X+Y+Z=0$, $X-Y-2Z=0$,
$2X-Y+Z=0$, and $X-3Y+2Z=0$, has rank 2.
\end{prop}

Just computing the characteristic polynomial of
Frobenius at $p=19$ seemed infeasible and so we do not know the splitting
behavior of $J$.

\section{Examples in the literature for genus higher than 1}
\label{lit}

In \cite{Sc2},
a 2-Selmer group is used to show that the Mordell-Weil rank over ${\bf Q}$
of the Jacobian of $y^{2}=f(x)$, where
$f(x)=x^{5}+16x^{4}-274x^{3}+817x^{2}+178x+1$, is 7. Let $L=
{\bf Q}[T]/(f(T))$ and ${\rm Cl}(L)$ denote the class group of the field $L$.
The fact that ${\rm dim}\, {\rm C}(L)/{\rm Cl}(L)^{2}$ is 4
was exploited.
In \cite{Sc1}, a 2-Selmer group is used to show that the Mordell-Weil rank
of the Jacobian of
$y^{2}=x(x-2)(x-3)(x-4)(x-5)(x-7)(x-10)$ over ${\bf Q}$ is 2.
In \cite{St2}, Stoll computed both 2-Selmer groups and $(1-\zeta_5)$-Selmer
groups for the Jacobians of some curves of the form $y^2=x^5+k$. Using
information from both, some were shown to have non-trivial 2-parts of
their Tate-Shafarevich groups.

In each of these cases, the hyperelliptic curve is of the form
described in Section~\ref{ytop}, namely $y^{2}=f(x)$
where $f$ has odd degree. As discussed in Section~\ref{ext},
there is a way of bounding the Mordell-Weil rank of the Jacobians of
hyperelliptic curves of the form $y^2=f(x)$, where $f$ has even degree;
see \cite{Ca,FPS,PS}.
The author has used this to show that the Mordell-Weil
rank over ${\bf Q}$
of the Jacobian of a curve of Colin Stahlke's given by $y^{2}=f(x)$
where $f(x)=
121x^6 - 138x^5 + 183x^4 + 370x^3 + 104x^2 - 112x + 1$
is exactly 12. Let $L={\bf Q}[T]/(f(T))$ and
${\rm Cl}(L)$ denote the class group of the field $L$. The fact that
${\rm dim}\, {\rm Cl}(L)/{\rm Cl}(L)^{2}$ is 9 was exploited.
In $\cite{FPS}$, the Mordell-Weil rank over ${\bf Q}$ of the Jacobian
of $y^{2}=x^{6}+8x^{5}+22x^{4}+22x^{3}+5x^{2}+6x+1$ is shown to be 1.
This is used to show that there are no ${\bf Q}$-rational quadratic
polynomials with rational periodic points of period 5.
In \cite{PS} the algorithm is extended further to curves of the
form $y^p = f(x)$ where the prime $p$ divides the degree of $f$.
In addition, the Mordell-Weil rank over ${\bf Q}$ of the Jacobian
of $y^{3}=(x^{2}-x+6)^{2}(x^{8}+3x+3)$
is shown to be 2. This example required working in a number field of
degree 16 over ${\bf Q}$.

Flynn has a technique for bounding the Mordell-Weil
rank of the Jacobian $J$ of a hyperelliptic
curve $C$ of genus 2 over a number field $K$, that
is the best one available for certain cases
(see \cite{Fl,CF}). Let us describe how it fits into our framework
and how it can be extended.
Assume $J[2]$ has a rational subgroup of order 4 which is isotropic
with respect to the 2-Weil pairing.
In most cases, the quotient
of $J$ by that subgroup is again the Jacobian $J'$ of a genus 2 curve $C'$.
The induced isogeny is called a Richelot's isogeny.
We denote it by $\phi$.
There is similarly a Richelot's isogeny $\phi'$ from $J'$
to $J$ such that $\phi'\circ\phi = 2$.
Because of the isotropy, $\lambda^{-1}\hat{J}[\hat{\phi'}]=J[\phi]$
for the canonical principal polarization $\lambda$ of $J$
with respect to $C$ (see \cite[prop.\ 16.8]{Mi2}).
In \cite{CF}, Cassels and Flynn present a method for computing
$S^{\phi}(K,J)$ and $S^{\phi'}(K,J')$ assuming that all elements of
$J[\phi]$ are rational.
They use the method described in Section 2.
The kernel of a Richelot's isogeny is isomorphic to
$V_{4}$, the Klein-4 group. The group
$H^{1}(G,V_{4})$ is trivial for all $G\subseteq {\rm Aut}(V_{4})$.
Thus for all possible Galois actions on $J[\phi]$ or $J'[\phi']$,
Assumption II holds and we can do a descent using a Richelot's isogeny.

There are a few examples in the literature like those in
Section~\ref{ytop} where $p\neq 2$.
In \cite{KS}, the Mordell-Weil rank
of the Jacobian of $y^{3}=x^{4}-1$ over ${\bf Q}(\zeta_{12})$ is shown
to be 0.
Fadeev and McCallum describe a map $F$
for quotients of the $p$th
Fermat curve given by $y^{p}=x^{a}(1-x)^{b}$ with $0 < a,b < p$;
see \cite{Fd,Mc}.

\pagebreak

\end{document}